\documentclass{article}
\usepackage{amsmath,amssymb,amsthm,nicefrac,braket}
\usepackage{graphicx,graphics,color}
\usepackage{bm}
\usepackage{enumerate}
\usepackage{url}

\newtheorem{proposition}{Proposition}[section]

\newtheorem{theorem}{Theorem}[section]
\newtheorem{lemma}{Lemma}[section]

\theoremstyle{definition}
\newtheorem{remark}{Remark}[section]
\newtheorem{definition}{Definition}[section]

\theoremstyle{definition}

\newtheoremstyle{blabla}
{3pt}
{3pt}
{}
{}
{\scshape}
{}
{.5em}
{}
\theoremstyle{blabla}
\newtheorem{assumption}{Assumption}[section]
\newtheorem{problem_b}{Problem}[section]

\newcommand{\dell}{\partial}
\newcommand{\dellmat}{\dell^{\bullet}}
\newcommand{\ale}{\mathcal{A}}
\newcommand{\dellale}{\dell^{\mathcal{A}}} 
\newcommand{\nablamat}{\nabla_{\Gamma(t)}}

\DeclareMathOperator{\dist}{dist}

\newcommand{\normt}[2]{\lvert #1 \rvert_{#2}}
\newcommand{\normbig}[2]{\bigl\lvert #1 \bigr\rvert_{#2}}

\newcommand{\Normstar}[2]{\lVert #1 \rVert_{*,#2}}

\newcommand{\B}{B}                          
\newcommand{\normal}{^{\textnormal{N}}}     
\newcommand{\pr}{\textnormal{Pr}}           
\newcommand{\prh}{\textnormal{Pr}_h}        
\newcommand{\wein}{\mathcal{H}}             

\newcommand{\Wh}{W_h}

\newcommand{\amat}{\partial^{\mathcal{A}}}
\newcommand{\amath}{\partial^{\mathcal{A}}_h}
\newcommand{\Btensor}{\mathcal{B}}
\newcommand{\diff}{\frac{\d}{\d t}}
\newcommand{\Ga}{\Gamma}
\newcommand{\Gat}{\Gamma(t)}
\newcommand{\GT}{\mathcal{G}_T}
\newcommand{\laplace}{\Delta}
\newcommand{\nbg}{\nabla_{\Gamma}}
\newcommand{\nbgh}{\nabla_{\Gamma_h}}
\newcommand{\mat}{\partial^{\bullet}}
\def \P {\mathcal{P}_h}
\newcommand{\Pt}{\widetilde{\mathcal{P}}_h}

\newcommand{\co}{continuous}
\def \d {\mathrm{d}}
\newcommand{\disp}{\displaystyle}
\newcommand{\eps}{\varepsilon}

\newcommand{\inv}{^{-1}}
\newcommand{\la}{\langle}

\newcommand{\nb}{\nabla}
\newcommand{\Om}{\Omega}
\newcommand{\op}{operator}
\newcommand{\pa}{\partial}
\newcommand{\R}{\mathbb{R}}
\newcommand{\ra}{\rangle}
\newcommand{\resp}{respectively}
\newcommand{\spn}{\textnormal{span}}
\newcommand{\st}{such that}

\def \t {(t)}
\newcommand{\Th}{\mathcal{T}_h}
\def \to {\rightarrow}
\newcommand{\vphi}{\varphi}
\def \nu {\textnormal{n}}
\begin{document}

\title{Higher--oder time discretizations with ALE finite elements for parabolic problems on evolving surfaces\footnote{This preprint was submitted to the IMA Journal of Numerical Analysis.}}

\author{Bal{\'a}zs Kov{\'a}cs\footnote{MTA-ELTE NumNet Research Group, P{\'a}zm{\'a}ny P. s{\'e}t{\'a}ny 1/C., 1117 Budapest, Hungary; e-mail address: koboaet@cs.elte.hu}, Christian Andreas Power Guerra\footnote{Mathematisches Institut, University of T\"{u}bingen, Auf der Morgenstelle 10., D-72076 T\"{u}bingen Germany; e-mail address: power@na.uni-tuebingen.de}}

\date{13. July 2014}


\maketitle

\begin{abstract}
{A linear evolving surface partial differential equation is first discretized in space by an arbitrary \linebreak Lagrangian Eulerian (ALE) evolving surface finite element method, and then in time either by a Runge--Kutta method, or by a backward difference formula. The ALE technique allows to maintain the mesh regularity during the time integration, which is not possible in the original evolving surface finite element method. Unconditional stability and optimal order convergence of the full discretizations is shown, for algebraically stable and stiffly accurate Runge--Kutta methods, and for backward differentiation formulae of order less than $6$. Numerical experiments are included, supporting the theoretical results.}
{full convergence, evolving surfaces, ESFEM, ALE, Runge--Kutta methods, BDF.}
\end{abstract}

\section{Introduction}

There are various approaches to solve parabolic problems on evolving surfaces. A starting point of the finite element approximation to (elliptic) surface partial differential equations is the paper of \cite{Dziuk88}, later this theory was extended to general parabolic equations on stationary surfaces by \cite{DziukElliott_SFEM}. They introduced the \emph{evolving surface finite element method} (ESFEM) to discretize parabolic partial differential equations on moving surfaces, c.f.\ \cite{DziukElliott_ESFEM}. They also gave optimal order error estimates in the $L^2$-norm, see \cite{DziukElliott_L2}. There is a survey type article by \cite{DziukElliott_acta}, which also serves as a rich source of details and references.

\medskip
Dziuk and Elliott also studied fully discrete methods, see e.g.\ \cite{DziukElliott_fulldiscr}. The numerical analysis of convergence of full discretizations with high order time integrators was first studied by \cite{DziukLubichMansour_rksurf}. They proved optimal order convergence for the case of algebraically stable implicit Runge--Kutta methods, and \cite{LubichMansourVenkataraman_bdsurf} proved optimal convergence for
backward differentiation formulae (BDF).

The ESFEM approach and convergence results were later extended to wave equations on evolving surfaces by \cite{LubichMansour_wave} and \cite{Mansour_GRK}. A unified presentation of ESFEM for parabolic problems and wave equations is given in \cite{diss_Mansour}.

\bigskip
As it was pointed out by Dziuk and Elliott, \textit{''A drawback of our method is the possibility of de\-ge\-ne\-rating grids. The prescribed velocity may lead to the effect, that the triangulation $\Ga_h(t)$ is distorted''}\footnote{quoted from Gerhard Dziuk and Charles M.\ Elliott from \cite{DziukElliott_ESFEM} Section 7.2}. To resolve this problem \cite{ElliottStyles_ALEnumerics} proposed an \emph{arbitrary Lagrangian Eulerian} (ALE) ESFEM approach, which in contrast to the (pure Lagrangian) ESFEM method, allows the nodes of the triangulation to move with a velocity which may not be equal to the surface (or material) velocity. They presented numerous examples where smaller errors can be achieved using a \emph{good} mesh.

Recently \cite{ElliottVenkataraman_ALEdiscrete} proved optimal order error bounds for the ALE ESFEM space discrete problems, and error bounds for the fully discrete schemes for the first and second-order backward differentation formulae. They also give numerous numerical experiments.

Arbitrary Lagrangian Eulerian FEM for moving domains were investigated by \cite{formaggia_nobile_alefem}. They also suggest some possible ways to define the new mesh if the movement of the boundary is given. \cite{nochetto_error,nochetto_stability} proved a-priori error estimates and time stability in a dis\co\ Galerkin setting.

\bigskip
This paper extends the convergence results of \cite{DziukLubichMansour_rksurf} for the Runge--Kutta discretizations, and the results of \cite{LubichMansourVenkataraman_bdsurf} the backward differentiation formulae, to the ALE framework and hence proves convergence of the fully discrete method suggested by \cite{ElliottStyles_ALEnumerics}.

We prove unconditional stability and convergence of these higher--order time discretizations, and also their optimal order convergence as a full discretization for evolving surface linear parabolic PDEs when coupled with the arbitrary Lagrangian Eulerian evolving surface finite element method as a space discretization. First, this is proved for stiffly accurate algebraically stable implicit Runge--Kutta methods (having the Radau IIA methods in mind).  Second, for the $k$-step backward differentiation formulae up to order five. Because of the lack of A-stability of the BDF methods of order greather than two, our proof requires a different techique than \cite{ElliottVenkataraman_ALEdiscrete}. Our results for BDF 1 and BDF 2 are matching theirs.

In the presentation we focus on the main differences compared to the previous results, and put less emphasis on those parts where minor modifications of the cited proofs are sufficient.

Our results are also true for the case of moving domains, however we will mostly stick to the evolving surface terminology.

\medskip
This paper is organised as follows.
In Section \ref{section_ALE_intro} we formulate the considered evolving surface parabolic problem, and describe the concept of arbitrary Lagrangian Eulerian methods together with other basic notions. The ALE weak formulation of the problem is also given.
In Section \ref{section_ALE FEM} we define the mesh approximating our moving surface and derive the semidiscrete version of the ALE weak form, which is equivalent to a system of ODEs. We also derive the ODE system resulting from a moving domain problem, which has the same properties. Then we recall some properties of the evolving matrices, and some estimates of bilinear forms. We also prove the analogous estimate for the new term appearing in the ALE formulation.
In Section \ref{sec:error-estim-impl} we prove stability of high order Runge--Kutta (R--K) methods applied to the ALE ESFEM semidiscrete problem, while Section \ref{section_BDF} is devoted to the corresponding results for the BDF methods.
Section \ref{section_errorbounds} contains the main results of this paper: the fully discrete methods, ALE ESFEM together with R--K or BDF method, have an unconditional and optimal order convergence both in space and time.
Finally, in Section \ref{section_numerics} we present numerical experiments, to illustrate our theoretical results.

\section{The arbitrary Langrangian Eulerian approach for evolving surface PDEs}
\label{section_ALE_intro}

In the following we consider an evolving closed hypersurface $\Ga\t$, $0 \leq t \leq T$, which moves with a given smooth velocity $v$.  Let $\mat u = \pa_{t} u + v \cdot \nbg u$
denote the material derivative of the function $u$, where $\nbg$ is the tangential gradient
given by $\nbg u = \nb u -\nb u \cdot \nu \nu$, with unit normal $\nu$. We denote by $\laplace_\Ga = \nbg\cdot\nbg$ the Laplace--Beltrami operator.  \par
We consider the following linear problem derived by \cite{DziukElliott_ESFEM}:
\begin{equation}\label{eq_ES-PDE-strong-form}
  \begin{cases}
    \phantom{.}\mat u(x,t) + u(x,t) \nb_{\Gat} \cdot v(x,t) - \laplace_{\Gat}
    u(x,t) = f(x,t)  & \textrm{ on } \Ga\t ,\\
    \phantom{\mat u(x,t) + u(x,t) \nb_{\Gat} \cdot v(x,t) - \laplace_{\Gat}
      x}u(x,0) = u_0(x) & \textrm{ on } \Ga(0).
  \end{cases}
\end{equation}
Basic and detailed references on evolving surface PDEs are \cite{DziukElliott_ESFEM,DziukElliott_acta,DziukElliott_L2} and \cite{diss_Mansour}.

For simplicity reasons we set in all chapters $f=0$, since the extension of our results to the inhomogeneous case are straightforward.

An important tool is the Green's formula (on closed surfaces), which takes the form
\begin{equation*}
    \int_\Ga \nbg z \cdot \nbg \phi = -\int_\Ga (\laplace_{\Ga} z) \phi
\end{equation*}
Finally, $\GT$ denotes the space--time surface, i.e.\ $\GT:= \cup_{t\in[0,T]}
\Ga\t\times\{t\}$.  We assume that $\GT\subset \R^{d+2}$ is a smooth
hypersurface (with boundary $\dell \GT = \big(\Ga(0) \times \{0\}\big) \cup \big(\Ga(T) \times \{T\}\big)$).

\bigskip
The weak formulation of this problem reads as
\begin{definition}[weak solution, \cite{DziukElliott_ESFEM} Definition 4.1]\label{def_ES-PDE-weak}
  A function $u\in H^1(\GT)$ is called a \emph{weak solution} of
  \eqref{eq_ES-PDE-strong-form}, if for almost every $t\in[0,T]$
    \begin{equation}\label{eq_ES-PDE-weak-form}
        \diff \int_{\Gat}\!\!\!\! u \vphi + \int_{\Gat}\!\!\!\! \nb_{\Gat} u \cdot \nb_{\Gat} \vphi = \int_{\Gat}\!\!\!\! u\mat \vphi
    \end{equation}
    holds for every  $\vphi \in H^1(\GT)$ and $u(.,0)=u_0$.
\end{definition}
For suitable $f$ and $u_{0}$ existence and uniqueness results, for the strong and the weak problem,
were obtained by \cite{DziukElliott_ESFEM}.

\subsection{The ALE map and ALE velocity and the corresponding weak form}


We assume that for each $t \in [0,T]$, $T>0$, $\Gamma^{m}\t\subset \R^{m+1}$
is an closed surface.  We call a subset $\Gamma^{m}\subset \R^{m+1}$ a
\emph{closed surface}, if $\Gamma$ is an oriented compact submanifold of
codimension $1$ without boundary.  We assume that there exists a smooth map
$\nu\colon \GT \to \R^{m+1}$ such that for each $t$ the restriction
\[
\nu_{t}\colon \Gamma(t)\to \R^{m+1},\quad \nu_{t}(x) := \nu( x,t)
\]
is the smooth normal field on $\Gamma(t)$.

\bigskip Now we shortly recall the surface description by diffeomorphic
parametrization, also used by \cite{DziukElliott_ESFEM}, and by
\cite{nochetto_stability}. An other important representation of the surface is
based on a signed distance function. For this we refer to
\cite{DziukElliott_ESFEM} (it is also described later in Section
\ref{subsection_lift}).

We assume that there exists a smooth map $ \Phi\colon \Gamma(0)\times [0,T] \to
\R^{m+1} $ which we call a \emph{dynamical system} or \emph{diffeomorphic
  parametrization} satisfying that
\[
\Phi_{t} \colon \Gamma(0) \to \Gamma\t, \qquad \Phi_{t}(y) := \Phi(y,t)
\]
is a diffeomorphism for every $t\in[0,T]$.  $(\Phi_{t})$ is called the \emph{flow} of
$\Phi$.  We observe:
\begin{itemize}
\item If $F\colon U\subset \R^{m}\to \Gamma(0)$ is a smooth parametrization of
  $\Gamma(0)$ then $F_{t}:= \Phi_{t}\circ F$ is a smooth parametrization of $\Gamma(t)$,
  hence the name diffeomorphic parametrization.
\item If we interpret $\Gamma(0)\times [0,T] \subset \R^{m+2}$ as a
  hypersurface, then $\Phi$ gives rise to a (submanifold) diffeomorphism
  \[
  \widetilde{\Phi}\colon \Gamma(0) \times [0,T] \to \mathcal{G}_{T},\quad
  \widetilde{\Phi}(y,t) := \bigl(\Phi_{t}(y),t\bigr).
  \]
\end{itemize}

\noindent The dynamical system $\Phi$ defines a (special) vector field $v$ and (special) time
derivative $\dellmat$ as follows: consider the differential equation (for $\Phi$)
\begin{equation}
  \label{eq_surface-velocity}
  \pa_t \Phi(\, .\, ,t) = v\bigl(\Phi(\, .\, ,t),t\bigr), \qquad \Phi(\, .\,
  ,0)=\textrm{Id}.
\end{equation}

\noindent The unique vector field $v$ is called the \emph{velocity of the
  surface evolution}, or the \emph{material velocity}. We assume, that the
material velocity is the same velocity as in
problem~\eqref{eq_ES-PDE-strong-form}.  It has the normal component $v\normal$.

\par
The time derivative $\dellmat$ is defined as follows (see e.g.\ \cite{DziukElliott_ESFEM} Section
2.2 or \cite{nochetto_stability} Section 1): for smooth $f
\colon \mathcal{G}_T \to \R$ and $x\in \Gamma\t$, such that $y\in \Gamma(0)$ for
which $\Phi_{t}(y)=x$, the \emph{material derivative} is defined as

\begin{equation}\label{eq_surface-derivative}
  \disp \mat f (x,t) := \left.\frac{\mathrm{d}}{\mathrm{d}t}\right\rvert_{(y,t)}
  f\circ \widetilde{\Phi}.
\end{equation}

\noindent Suppose that $f$ has a smooth extension $\bar{f}$ in an open neighborhood of
$\Gamma\t$, (\cite{DziukElliott_acta} has shown how to use the oriented distance
function to construct such extensions), then by the chain rule they obtained the
following identity for the material derivative:
\[
\disp \mat f (x,t) = \left.\frac{\dell \bar{f}}{\dell t}\right\rvert_{(x,t)} +
v(x,t) \cdot \nabla \bar{f} (x,t),
\]
which is clearly independent of the extension by \eqref{eq_surface-derivative}.

\begin{remark}
  An evolving surface $\Gamma(t)$ generally posses many different dynamical
  systems.  Consider for example the (constant) evolving surface
  $\Gamma(t)=\Gamma(0)=S^{m}\subset \R^{m+1}$ with the two (different) dynamical
  system $\Phi(x,t)=x$ and $\Psi(x,t)=\alpha(t)x$, where $\alpha\colon [0,T]\to
  O(m+1)$ is a smooth curve in the orthogonal matrices.
\end{remark}

\begin{definition}
    Let $\ale \neq \Phi$ be any other dynamical system for $\Gamma(t)$. It is
    called an \emph{arbitrary   Lagrangian Eulerian map} (ALE map). The
    associated velocity will be denoted by $w$, which we refer as the \emph{ALE velocity} and finally $\dellale$ denotes the ALE material derivative.
  \end{definition}

One can show that for all $t\in [0,T]$ and
$x\in \Gamma(t)$
\begin{equation}\label{eq_v-w_tangent}
    v(x,t) - w(x,t) \quad \text{is a tangential vector.}
\end{equation}

\bigskip
The formula for the differentiation of a parameter-dependent surface integral played a decisive role in the analysis of evolving surface problems. In the following lemma we will state its ALE version, together with the connection between the material derivative and ALE material derivative.
\begin{lemma}
    Let $\Ga\t$ be an evolving surface and $f$ be a function defined in $\mathcal{G}_T$, such that all the following quantities exist.
  \begin{enumerate}[(a)]
      \item (Leibniz formula \cite{DziukElliott_ESFEM}/ Reynolds transport identity
        \cite{nochetto_stability}) There holds
        \begin{align}
          \label{eq_leibniz}
          \frac{\mathrm{d}}{\mathrm{d}t}\int_{\Gamma\t}f  = \int_{\Gamma\t}
          \dellale f + f  \ \nabla_{\Gamma\t}\cdot w. 
        \end{align}
      \item There also holds
        \begin{align}
          \label{eq_mat-vs-ALE-mat}
          \dellale f &= \dellmat f + (w-v)\cdot \nablamat f.
        \end{align}
  \end{enumerate}
\end{lemma}

\begin{proof}

At first we prove (b): consider an extension $\bar{f}$ of $f$. Use the chain rule for $\dellale f$ and $\dellmat f$ and note the identity (c.f.\ \eqref{eq_v-w_tangent})
  \[
  (w(.,t)-v(.,t))\cdot \nabla \bar{f}(.,t) = (w(.,t)-v(.,t))\cdot \nabla_{\Gamma} f(.,t).
  \]

  To prove (a) use the original Leibniz formula from \cite{DziukElliott_ESFEM}:
  \begin{align*}
    \frac{\mathrm{d}}{\mathrm{d}t}\int_{\Gamma}f  = \int_{\Gamma}
    \dellmat f + f \ \nabla_{\Gamma}\cdot v.
  \end{align*}
  Now use (b) and Greens identity for surfaces to complete the proof.
\end{proof}

\bigskip
Now we have everything at our hands to derive the ALE version of the weak form of the evolving surface PDE \eqref{eq_ES-PDE-strong-form}.
\begin{lemma}[ALE weak solution]\label{lemma_ALE-ES-PDE-weak}
    The arbitrary Lagrangian Eulerian weak solution for an evolv\-ing surface partial differential equation is a function $u\in H^1(\GT)$, if for almost every $t\in[0,T]$
    \begin{equation*}
        \diff \int_{\Gat}\!\!\!\! u \vphi + \int_{\Gat}\!\!\!\! \nb_{\Gat} u \cdot \nb_{\Gat} \vphi
        + \int_{\Gat}\!\!\!\!  u (w-v) \cdot \nb_{\Gat} \vphi = \int_{\Gat}\!\!\!\! u\amat \vphi
    \end{equation*}
    holds for every  $\vphi \in H^1(\GT)$ and $u(\, .\, ,0)=u_0$.  If $u$ solves
    equation~\eqref{eq_ES-PDE-weak-form} then $u$ is an ALE weak solution.
\end{lemma}

\begin{proof}
    We start by substituting the material derivative by the ALE material
    derivative in \eqref{eq_ES-PDE-weak-form}, using the relation
    \eqref{eq_mat-vs-ALE-mat}, connecting the different material derivatives
    (c.f.\ \eqref{eq_v-w_tangent}), i.e.\ by putting
    \begin{equation*}
        \mat \vphi = \amat \vphi + (v-w) \cdot \nbg \vphi
    \end{equation*}
    into the weak form, and rearranging the terms, we get the desired formulation.
\end{proof}

\section{The ALE finite element discretization}
\label{section_ALE FEM}

This section is devoted to the spatial semidiscretization of the parabolic moving surface PDE with the ALE version of the evolving surface finite element method, the ESFEM was developed by \cite{DziukElliott_ESFEM}. In the original case the nodes were moving only with the material velocity along the surface, which could lead to degenerated meshes.
One can maintain the good properties of the initial mesh by having additional tangential velocity

The ALE ESFEM discretization will lead to a system of ordinary differential equations (ODEs) with time dependent matrices. We will prove basic properties of those matrices, which will be one of our main tools to prove stability of time discretizations and convergence of full discretizations. We will also recall the lifting \op\ and its properties introduced by \cite{DziukElliott_ESFEM}, which enables us to compare functions from the discrete and \co\ surface.

\subsection{ALE finite elements for evolving surfaces}
\label{subsection_ALE-ESFEM}

First, the initial surface $\Gamma(0)$ is approximated by a triangulated one denoted
by $\Ga_h(0)$, which is given as
\begin{equation*}
    \Ga_h(0) := \bigcup_{E(0)\in \Th(0)} E(0).
\end{equation*}
Let $a_i(0)$, ($i=1,\ldots, N$), denote the initial nodes lying on the initial
\co\ surface.  Now the nodes are evolved with respect to the ALE map $\ale$,
i.e.\ $a_{i}(t):= \ale\bigl(a_{i}(0),t\bigr)$.  Obviously they remain on the
\co\ surface $\Gat$ for all $t$.  Therefore the smooth surface $\Gat$ is
approximated by the triangulated one denoted by $\Ga_h(t)$, which is given as
\begin{equation*}
    \Ga_h(t) := \bigcup_{E(t)\in \Th(t)} E(t).
\end{equation*}
We always assume that the (evolving) simplices $E(t)$ are forming an admissible
triangulation (c.f.\ \cite{DziukElliott_ESFEM}) $\Th(t)$ with $h$ denoting the maximum
diameter.

\medskip
The discrete tangential gradient on the discrete surface $\Ga_h\t$ is given by
\begin{equation*}
    \nb_{\Ga_h\t} f := \nb f - \nb f \cdot \nu_h \nu_h = \prh (\nb f),
\end{equation*}
understood in a piecewise sense, with $\nu_h$ denoting the normal to $\Ga_h(t)$
and $\prh := I - \nu_{h} \nu_{h}^{T} $.

\bigskip
For every $t\in[0,T]$ we define the finite element subspace
\begin{equation*}
    S_h\t := \big\{ \phi_h \in C(\Ga_h\t) \, \, \big| \, \,  \phi_h|_E \textrm{ is linear, for all } E\in \Th\t \big\}.
\end{equation*}

\noindent The moving basis functions $\chi_j$ are defined as $\chi_j(a_i\t,t) = \delta_{ij}$ for all $i,j = 1, 2, \dotsc, N$, and hence
\begin{equation*}
    S_h\t = \spn\big\{ \chi_1( \, . \,,t), \chi_2( \, . \,,t), \dotsc, \chi_N( \, . \,,t) \big\}.
\end{equation*}

\noindent We continue with the definition of the interpolated velocities on the discrete surface $\Ga_h\t$:
\begin{equation*}
    V_h( \, . \,,t) = \sum_{j=1}^N v(a_j\t,t) \chi_j( \, . \,,t), \qquad \quad
    \Wh( \, . \,,t) = \sum_{j=1}^N w(a_j\t,t) \chi_j( \, . \,,t)
\end{equation*}
are the discrete velocity, and the discrete ALE velocity, \resp. The discrete material derivative, and its ALE version is given by
\begin{equation*}
    \mat_h \phi_h = \pa_t \phi_h + V_h \cdot \nb \phi_h, \qquad \quad
    \amath \phi_h = \pa_t \phi_h + \Wh \cdot \nb \phi_h.
\end{equation*}

\bigskip
In this setting the key \textit{transport property} derived by \cite{DziukElliott_ESFEM} Proposition 5.4, is the following
\begin{equation}\label{eq_ES-transport-prop}
    \amath \chi_k = 0 \qquad \textrm{for} \quad k=1,2,\dotsc,N.
\end{equation}

\bigskip
The spatially discrete ALE problem for evolving surfaces is formulated in
\begin{problem_b}[Semidiscretization in space]\label{problem_ESFEM-semidiscrete}
    Find $U_h\in S_h\t$ \st\
{\setlength\arraycolsep{.13889em}
\begin{eqnarray*}
        \diff && \!\!\int_{\Ga_h\t}\!\! U_h \phi_h
        + \int_{\Ga_h\t}\!\! \nabla_{\Gamma_h\t} U_h  \cdot \nabla_{\Gamma_h\t}\phi_h \\
        + && \!\!\int_{\Ga_h\t}\!\! U_h (\Wh-V_h)\cdot \nabla_{\Gamma_h\t}\phi_h =
         \int_{\Ga_h\t}\!\! U_h \amath \phi_h,  \qquad (\forall \phi_h \in S_h\t),
\end{eqnarray*}}
    with the initial condition $U_h( \, . \,,0)=U_h^0\in S_h(0)$ is a sufficient approximation to $u_0$.
\end{problem_b}

The ODE form of the above problem can be derived by setting
\begin{equation*}
    U_h( \, . \,,t) = \sum_{j=1}^N \alpha_j\t \chi_j( \, . \,,t)
\end{equation*}
and $\phi_h=\chi_j$ and using the transport property for evolving surfaces \eqref{eq_ES-transport-prop}.

\begin{proposition}[ODE system for evolving surfaces]\label{prop_ODE-system-ES}
    The spatially semidiscrete problem is equivalent to the ODE system for $\alpha\t=(\alpha_j\t)\in\R^N$
    \begin{equation}\label{eq_ES-ODE}
        \disp
        \begin{cases}
            \disp\diff \big(M\t \alpha\t\big) + A\t \alpha\t + \B\t \alpha\t = 0 \\
            \disp\phantom{\diff \big(M\t \alpha\t\big) + A\t \alpha\t + \B\t }\alpha(0) = \alpha_0
        \end{cases}
    \end{equation}
    where $M\t$ and $A\t$ are the evolving mass and stiffness matrices defined as
    \begin{equation*}
        M(t)_{kj} = \int_{\Ga_h\t}\!\!\!\! \chi_j \chi_k, \qquad A(t)_{kj} = \int_{\Ga_h\t}\!\!\!\! \nb_{\Ga_h\t}\chi_j \cdot \nb_{\Ga_h\t } \chi_k,
    \end{equation*}
    and the evolving matrix $\B\t$ is given by
    \begin{equation}\label{eq_new-term-W}
        \B(t)_{kj} = \int_{\Ga_h\t}\!\!\!\! \chi_j (\Wh-V_h)\cdot \nb_{\Ga_h\t } \chi_k.
    \end{equation}
\end{proposition}
The proof of this proposition is analogous to the corresponding one by \cite{DziukLubichMansour_rksurf}.

\begin{remark}
  In the original ESFEM setting there was no direct involvement of velocities, but in the ALE formulations there is. We remark here that since the normal components of the ALE and material velocity are equal, during computations one can work only with the difference of the two velocities, i.e.\ the additional tangential component of the ALE velocity. We only keep the above formulation to leave the presentation plain and simple.
\end{remark}

\subsection{ALE finite elements for moving domains}
\label{subsection_ALE-MDFEM}

However our main interest is evolving surface PDEs, our results are also valid for moving domain partial differential equations. We will see that the corresponding ODE system of ALE finite element semidiscretization of such problems are coinciding with the ODE problem for evolving surface PDEs, \eqref{eq_ES-ODE}. Therefore we shortly describe how to derive this system.

\bigskip
Let us consider the following parabolic partial differential equation over the rectangular moving domain $\Om(t)$:

\begin{equation}\label{eq_MD-PDE-strong}
    \begin{cases}
        \phantom{.}\mat u(x,t) + u(x,t) \nb \cdot v(x,t) - \laplace u(x,t) = f \phantom{u_0} \quad \textrm{ in } \Om\t,\\
        \phantom{\mat u(x,t) + u(x,t) \nb \cdot v(x,t) - \laplace x}u( \, . \,,t) = u_0 \phantom{f} \quad \textrm{ in } \Om(0),
    \end{cases}
\end{equation}
with homogeneous Dirichlet boundary conditions for all $t\in[0,T]$.

\bigskip
The moving domain FEM is defined just as usual, but the nodes are moving with the given ALE velocity: $\Th\t$ is an admissible triangulation of the moving domain $\Om\t$, with moving nodes $a_i\t$ for $t\in[0,T]$. Therefore we have for every $t\in[0,T]$ the finite element subspace $S_h\t$ consisting of piecewise linear functions, and
\begin{equation*}
    S_h\t = \spn\big\{ \chi_1( \, . \,,t), \chi_2( \, . \,,t), \dotsc, \chi_N( \, . \,,t) \big\},
\end{equation*}
where $\chi_j(a_i\t,t) = \delta_{ij}$, and vanishing at the boundary.

For domains the tangential gradient reduces to the usual gradient. The interpolated velocities of the discrete moving domain $\Om_h\t$, and hence the discrete material derivatives, are defined again by using the finite element interpolants. The \textit{transport property} is also remaining the same in the moving domain finite element setting.

\bigskip
The spatially discrete ALE problem for moving domains is formulated in:
\begin{problem_b}[Semidiscretization in space]\label{problem_MDFEM-semidiscrete}
    Find $U_h( \, . \,,t)\in S_h\t$ \st\
    \begin{eqnarray}\label{eq_MD ALE-spat-discrete}
        \diff \int_{\Om_h\t} U_h \phi_h + \int_{\Om_h\t} \nb U_h \cdot \nb \phi_h
        + \int_{\Om_h\t} U_h (\Wh-V_h)\cdot \nb \phi_h &=& \int_{\Om_h\t} U_h \amath\phi_h, \\
        & & \qquad (\forall \phi_h\in S_h\t)\nonumber
    \end{eqnarray}
    with the initial condition $U_h( \, . \,,0)=U_h^0\in S_h(0)$ is a sufficient approximation to $u_0$, by the definition of the basis functions $U_h( \, . \,,t)$ satisfies the homogeneous Dirichlet boundary condition.
\end{problem_b}

The ODE form of the above problem can be derived analogously, and yields:
\begin{proposition}[ODE system for moving domains]\label{prop_ODE-system-MD}
    The spatially semidiscrete problem \eqref{eq_MD ALE-spat-discrete} is equivalent to the ODE system for $\alpha\t=(\alpha_j\t)\in\R^N$
    \begin{equation}\label{eq_MD-ODE}
        \disp
        \begin{cases}
            \disp\diff \big(M\t \alpha\t\big) + A\t \alpha\t + \B\t \alpha\t = 0 \\
            \disp\phantom{\diff \big(M\t \alpha\t\big) + A\t \alpha\t + \B\t }\alpha(0) = \alpha_0
        \end{cases}
    \end{equation}
    where $M\t$ and $A\t$ are the evolving mass and stiffness matrices defined as
    \begin{equation*}
        M(t)_{kj} = \int_{\Om_h\t}\!\!\!\! \chi_j \chi_k, \qquad A(t)_{kj} = \int_{\Om_h\t}\!\!\!\! \nb \chi_j \cdot \nb \chi_k,
    \end{equation*}
    and the evolving matrix $\B\t$ is given by
    \begin{equation*}
        \B(t)_{kj} = \int_{\Om_h\t}\!\!\!\! \chi_j (\Wh-V_h)\cdot \nb \chi_k.
    \end{equation*}
\end{proposition}

\bigskip
\begin{remark}\label{remark_coincidence_of_ODEs}
    A very important point is, that formally the ODE problems for evolving surface and moving domain problems, \eqref{eq_ES-ODE} and \eqref{eq_MD-ODE}, are coincident. Furthermore, the crucial properties of the matrices are also the same for both cases. In the rest of the paper we will use the terminology of the evolving surface PDEs, but clearly our results hold for moving domain problems as well.
\end{remark}

\subsection{Properties of the evolving matrices}
Clearly the evolving stiffness matrix is symmetric, positive semi-definite and the mass matrix is symmetric, positive definite. Through the paper we will work with the norm and semi-norm introduced by \cite{DziukLubichMansour_rksurf}:
\begin{equation}\label{eq_normdefs}
    |z\t|_{M\t} = \|Z_h\|_{L^2(\Ga_h\t)} \qquad \textrm{and} \qquad  |z\t|_{A\t} = \|\nbgh Z_h\|_{L^2(\Ga_h\t)},
\end{equation}
for arbitrary $z\t\in \R^N$, where $Z_h( \, . \,,t)=\sum_{j=1}^N z_j\t \chi_j( \, . \,,t)$.

\bigskip
A very important lemma in our analysis is the following:
\begin{lemma}[\cite{DziukLubichMansour_rksurf} Lemma 4.1 and \cite{LubichMansourVenkataraman_bdsurf} Lemma 2.2]
    There are constants $\mu, \kappa$ (independent of $h$) \st
    {\setlength\arraycolsep{.13889em}
    \begin{eqnarray}
        \label{eq_mtxlemma-M} z^T \big( M(s) - M\t\big) y &\leq& (e^{\mu(s-t)}-1) |z|_{M\t}|y|_{M\t} \\
        \label{eq_mtxlemma-Minv} z^T \big( M\inv(s) - M\inv\t\big) y &\leq& (e^{\mu(s-t)}-1) |z|_{M\inv\t}|y|_{M\inv\t}  \\
        \label{eq_mtxlemma-A} z^T \big( A(s) - A\t\big) y &\leq& (e^{\kappa(s-t)}-1) |z|_{A\t}|y|_{A\t}
    \end{eqnarray}}
    for all $y,z\in \R^N$ and $s,t \in [0,T]$.
\end{lemma}
We will use this lemma with $s$ close to $t$, and then $(e^{\mu(s-t)}-1)\leq 2 \mu (s-t)$ holds. In particular for $y=z$ we have
\begin{align}
  \label{eq:11}  \normt{z}{M(s)}^{2}  \leq (1+2\mu   (t-s)) \normt{z}{M(t)}^{2}, \\
  \label{eq:11b} \normt{z}{A(s)}^{2}  \leq (1+2\kappa(t-s)) \normt{z}{A(t)}^{2}.
\end{align}

\bigskip
The following technical lemma will play a crucial role in this paper.
\begin{lemma}\label{lemma-B_estimate}
    Let $y,z\in \R^N$ and $t \in [0,T]$ be arbitrary, then
    \begin{equation}
        \big| \la \B\t z | y\ra \big| \leq c_{\ale} |z|_{M\t} |y|_{A\t},
    \end{equation}
    where the constant $c_{\ale}>0$ is depending only on the differences of the velocities, and independent of $h$.
\end{lemma}
\begin{proof}
    Using the definition of $\B$ (see \eqref{eq_new-term-W}) we can write
    \begin{equation*}
        \big| \la \B\t z | y\ra \big| = \Big|\int_{\Ga_h} Z_h (\Wh-V_h)\cdot \nbgh Y_h \Big| \leq \|\Wh-V_h\|_{L^{\infty}(\Ga_h\t)} \int_{\Ga_h} |Z_h| \ |\nbgh Y_h|,
    \end{equation*}
    then by applying the Cauchy--Schwarz inequality and using the equivalence of norms over the discrete and \co\ surface (c.f.\ \cite{DziukElliott_ESFEM}, Lemma 5.2), we obtain the stated result.
\end{proof}

\subsection{Lifting process}
\label{subsection_lift}

In the following we introduce the so called \emph{lift operator} which was introduced by \cite{Dziuk88} and further investigated by \cite{DziukElliott_ESFEM,DziukElliott_L2}. The lift
operator can be interpreted as a geometric projection: it projects a finite
element function $\varphi_{h}\colon \Ga_{h}\t \to \R$ on the discrete surface
$\Gamma_{h}(t)$ onto a function $\varphi_{h}^{l}\colon \Ga\t \to \R$ on the
smooth surface $\Ga\t$, therefore it is crucial for our error estimates.

\bigskip We assume that there exists an open bounded set $U(t)\subset \R^{m+1}$

such that $\dell U(t) = \Ga\t$. The \emph{oriented distance function} $d$ is
defined as
\[
\R^{m+1}\times [0,T]\to \R,\quad d(x,t) :=
\begin{cases}
  \dist\bigl(x, \Gamma(t)\bigr) & x \in \R^{m+1} \setminus U(t), \\
  - \dist\bigl(x, \Gamma(t)\bigr) & x \in U(t).
\end{cases}
\]
For $\mu >0 $ we define $\mathcal{N}(t)_{\mu} := \bigl\{ x\in
\R^{m+1} \mid \dist\bigl(x,\Gamma(t)\bigr) < \mu \bigr\}$.  Clearly
$\mathcal{N}(t)_{\mu}$ is an open neighborhood of $\Gamma(t)$.  \cite{GilbargTrudinger} in Lemma~14.16 have shown the following important regularity result about $d$.

\begin{lemma}
  Let $U(t) \subset \R^{m+1} $ be bounded and $\Gamma(t)\in C^{k} $ for $k\geq 2$.  Then
  there exists a positive constant $\mu$ depending on $U$ such that
  $d\in C^{k}\bigl(\mathcal{N}(t)_{\mu}\bigr)$.
\end{lemma}

\noindent \cite{GilbargTrudinger} also mentioned that $\mu^{-1}$ bounds
the principal curvatures of $\Gamma(t)$.

\medskip For each $x\in \Gamma(t)_{\mu}$ there exists a unique $p=
p(x,t)\in \Gamma(t)$ such that $\lvert x - p \rvert =
\dist\bigl(x,\Gamma(t)\bigr)$, then $x$ and $p$ are related by the important equation:
\begin{align}
  \label{eq:21}
  x = p + \nu(p,t) d(x,t).
\end{align}
We assume that $\Ga_{h}\t \subset \mathcal{N}\t$.  The \textit{lift operator}
$\mathcal{L}$ maps a \co\ function $\eta_h \colon \Ga_h\to\R$ onto a function
$\mathcal{L}(\eta_{h})\colon \Ga\to \R$ as follows: for every $x\in
\Ga_{h}\t$ exists via equation~\eqref{eq:21} an unique $p=p(x,t)$.  We set
pointwise
\begin{equation*}
  \mathcal{L}(\eta_h)(p,t):= \eta_{h}^{l}(p,t) := \eta_h(x,t).
\end{equation*}
It is clear that $\mathcal{L}(\eta_{h})$ is \co\ and that if $\eta_{h}$ has weak
derivatives then $\mathcal{L}(\eta_{h})$ also has weak derivatives.

\medskip We now recall some notions using the lifting process from \cite{Dziuk88,DziukElliott_ESFEM} and \cite{diss_Mansour} using the notations of the last reference. We have the lifted finite element space
\begin{equation*}
    S_h^l\t := \big\{ \vphi_h = \phi_h^l \, | \, \phi_h\in S_h\t \big\},
\end{equation*}
by $\delta_h$ we denote the quotient between the \co\ and discrete surface
measures, $\d A$ and $\d A_h$, defined as $\delta_h \d A_h = \d A$. Further, we recall that
\begin{equation*}
    \pr := \big(\delta_{ij} - \nu_{i}\nu_{j}\big)_{i,j=1}^N \quad \textrm{and} \quad \prh := \big(\delta_{ij} - \nu_{h,i}\nu_{h,j}\big)_{i,j=1}^N
\end{equation*}
are the projections onto the tangent spaces of $\Ga$ and $\Ga_h$. Finally $\wein$ ($\wein_{ij} = \pa_{x_j}\nu_i$) is the (extended) Weingarten map. For these quantities we recall some results from \cite{DziukElliott_ESFEM,DziukElliott_L2}.

\begin{lemma}[\cite{DziukElliott_ESFEM} Lemma 5.1 and \cite{DziukElliott_L2} Lemma 5.4]\label{lemma_geometric-est}
    Assume that $\Ga_h\t$ and $\Ga\t$ is from the above setting, then we have the estimates
    \begin{equation*}
        \|d\|_{L^\infty(\Ga_h)} \leq c h^2, \qquad \|\nu_j\|_{L^\infty(\Ga_h)} \leq c h, \qquad \|1-\delta_h\|_{L^\infty(\Ga_h)} \leq c h^2, \qquad \|(\amath)^{(\ell)} d \|_{L^\infty(\Ga_h)} \leq c h,
    \end{equation*}
    where $(\amath)^{(\ell)}$ denotes the $\ell$-th discrete ALE material derivative.
\end{lemma}
The second estimate can be found in the proof of the cited lemmata.

\subsection{Bilinear forms and their properties}

We use the time dependent bilinear forms defined by \cite{DziukElliott_L2}: for $z,\vphi \in H^1(\Ga)$, and their discrete analogs for $Z_h, \phi_h \in S_h$:
\begin{equation*}
    \begin{aligned}[c]
        a(z,\vphi)   &= \int_{\Ga\t} \nbg z \cdot \nbg \vphi, \\
        m(z,\vphi)   &= \int_{\Ga\t} z \vphi, \\
        g(w;z,\vphi) &= \int_{\Ga\t} (\nbg \cdot w) z\vphi,\\
        b(w;z,\vphi) &= \int_{\Ga\t} \Btensor(w) \nbg z \cdot \nbg \vphi,
    \end{aligned}
    \qquad \qquad
    \begin{aligned}[c]
        a_h(Z_h,\phi_h)     &= \sum_{E\in \Th} \int_E \nbgh Z_h \cdot \nbgh \phi_h, \\
        m_h(Z_h,\phi_h)     &= \int_{\Ga_h\t} Z_h \phi_h\\
        g_h(\Wh;Z_h,\phi_h) &= \int_{\Ga_h\t} (\nbgh \cdot \Wh) Z_h \phi_h,\\
        b_h(\Wh;Z_h,\phi_h) &= \sum_{E\in \Th} \int_E \Btensor_h(\Wh) \nbgh Z_h \cdot \nbgh \phi_h,
    \end{aligned}
\end{equation*}

where the discrete tangential gradients are understood in a piecewise sense, and
with the matrices

\begin{alignat*}{3}
    \Btensor(w)_{ij} &= \delta_{ij} (\nbg \cdot w) - \big( (\nbg)_i w_j + (\nbg)_j w_i \big), \qquad && (i,j=1,2,\dotsc,m),\\
    \Btensor_h(\Wh)_{ij} &= \delta_{ij} (\nbg \cdot \Wh) - \big( (\nbgh)_i (\Wh)_j + (\nbgh)_j (\Wh)_i \big), \qquad && (i,j=1,2,\dotsc,m).
\end{alignat*}

We will also use the transport lemma:
\begin{lemma}[\cite{DziukElliott_L2} Lemma 4.2]\label{lemma_transport-prop}
    For $z_h, \ \vphi_h, \ \amath z_h, \ \amath \vphi_h \in S_h^l\t \subset H^1(\Ga)$ we have:
    {\setlength\arraycolsep{.13889em}
    \begin{eqnarray*}
        \diff m(z_h,\vphi_h) &=& m(\amath z_h,\vphi_h) + m(z_h,\amath \vphi_h) + g(w_h;z_h,\vphi_h), \\
        \diff a(z_h,\vphi_h) &=& a(\amath z_h,\vphi_h) + a(z_h,\amath \vphi_h) + b(w_h;z_h,\vphi_h).
    \end{eqnarray*}}
\end{lemma}
Versions of this lemma with continuous non-ALE material derivatives, or discrete bilinear forms are also true, see e.g.\ \cite[Lemma 6.4]{diss_Mansour}.

\bigskip
We will need the following estimates between the \co\ and discrete bilinear forms.
\begin{lemma}[\cite{DziukElliott_L2} Lemma 5.5]\label{lemma_estimation-of-forms}
    For arbitrary $Z_h,\phi_h\in S_h\t$, with corresponding lifts $z_h,\vphi_h \in S_h^l\t$ we have the bound
    {\setlength\arraycolsep{.13889em}
    \begin{eqnarray*}
        \disp \big| m(z_h,\vphi_h) - m_h(Z_h,\phi_h) \big| &\leq& c h^2 \|z_h\|_{L^2(\Ga\t)} \|\vphi_h\|_{L^2(\Ga\t)}, \\
        \disp \big| a(z_h,\vphi_h) - a_h(Z_h,\phi_h) \big| &\leq& c h^2 \|\nbg z_h\|_{L^2(\Ga\t)} \|\nbg \vphi_h\|_{L^2(\Ga\t)}.
    \end{eqnarray*}}
\end{lemma}
Apart from the above crucial estimates of lifts and bilinear forms, we need an analogous estimate for the new term, represented by the matrix $\B\t$ (which is the result of the ALE approach).
\begin{lemma}\label{lemma_estiamtion-of-new-form}
    For arbitrary $Z_h,\phi_h\in S_h\t$, with corresponding lifts $z_h,\vphi_h \in S_h^l\t$ we have the bound
    \begin{equation*}
        \disp \big| m(z_h,(w_h-v_h)\cdot \nbg \vphi_h) - m_h(Z_h,(\Wh-V_h)\cdot \nbg \phi_h) \big|
        \leq c h^2 \|z_h\|_{L^2(\Ga\t)} \|\nbg\vphi_h\|_{L^2(\Ga\t)},
    \end{equation*}
    where the constant $c$ is only depending on the difference of the velocities, and the surface.
\end{lemma}

\begin{proof}
    We begin by recalling the connection between the discrete and \co\ tangential gradients (first derived by \cite{Dziuk88}):
    \begin{equation*}
        \nbgh \phi_h = \prh(I-d\wein) \nbg \vphi_h.
    \end{equation*}

    Using this, we can start estimating as follows
    {\setlength\arraycolsep{.13889em}
    \begin{eqnarray*}
        & & \Big| m(z_h,(w_h-v_h)\cdot \nbg \vphi_h) - m_h(Z_h,(\Wh-V_h)\cdot \nbgh \phi_h) \Big| \\
        &=& \Big| \int_\Ga z_h (w_h-v_h)\cdot \nbg \vphi_h \d A - \int_{\Ga_h} Z_h(\Wh-V_h)\cdot \nbgh \phi_h \d A_h \Big| \\
        &=& \Big| \int_\Ga z_h (w_h-v_h)\cdot\big(I - \frac{1}{\delta_h} \prh(I-d\wein)\big) \nbg \vphi_h) \d A \Big| \\
        &\leq&  \int_\Ga |z_h| |(w_h-v_h)| \  \Big( \frac{1}{\delta_h}|\delta_h I - \prh| +  \big|d (\frac{1}{\delta_h} \prh \ \wein)\big| \Big) |\nbg \vphi_h| \d A \\
        &\leq&  \int_\Ga |z_h| |(w_h-v_h)| \  \Big( \frac{1}{\delta_h}|\delta_h - 1| + \frac{1}{\delta_h}|\nu_{h}\nu_{h}^T| +  \big|d (\frac{1}{\delta_h} \prh \ \wein)\big| \Big) |\nbg \vphi_h| \d A \\
        &\leq& c h^2 \|z_h\|_{L^2(\Ga\t)} \|\nbg\vphi_h\|_{L^2(\Ga\t)}.
    \end{eqnarray*}}

    For the final estimate we have used that $(\prh)_{ij} = \delta_{ij} - \nu_{h,i}\nu_{h,j}$, further that $\nu_{h,j}=\pa_{x_j} d$ can be estimated by $c h$, and $1-\delta_h$ and $d$ can be estimated by $c h^2$, see Lemma \ref{lemma_geometric-est}, while the other terms are bounded, together with the fact that the norms on the discrete and \co\ surfaces are equivalent (see e.g.\ Lemma 5.2 by \cite{DziukElliott_ESFEM}).
\end{proof}

\section{Error estimates for implicit Runge--Kutta methods}
\label{sec:error-estim-impl}

We consider an $s$-stage implicit Runge--Kutta method (R--K) for the time discretization of the ODE system \eqref{eq_MD-ODE}, coming from the ALE ESFEM space discretization of the parabolic evolving surface PDE.

In the following we extend the stability result for R--K methods of \cite{DziukLubichMansour_rksurf}, Lemma 7.1, to the case of ALE evolving surface finite element method.
Apart form the properties of the ALE ESFEM the proof is based on the energy estimation techniques of \cite{LubichOstermann_RK} (Theorem 1.1).

\bigskip
For the convenience of the reader we recall the method: for simplicity, but not necessarily, we assume equidistant time steps $t_{n}:= n\tau$, with step size $\tau$. The $s$-stage implicit Runge--Kutta method reads
\begin{subequations}
  \begin{alignat}{3}
    \label{eq_rk-a}
    M_{ni} \alpha_{ni} &= M_{n} \alpha_{n} + \sum_{j=1}^{s} a_{ij}
    \dot{\alpha}_{nj}, \qquad &&\text{for} \quad i=1,2,\dotsc,s, \\
    \label{eq_rk-b}
    M_{n+1} \alpha_{n+1} &= M_{n} \alpha_{n} + \sum_{i=1}^{s} b_{i}
    \dot{\alpha}_{ni},&&
    \intertext{where the internal stages satisfy}
    \nonumber
    0 &= \dot{\alpha}_{ni} + \B_{ni} \alpha_{ni} + A_{ni} \alpha_{ni} \qquad &&\text{for} \quad i=1,2,\dotsc,s,
  \end{alignat}
\end{subequations}

\noindent with $A_{ni}:= A(t_{n}+c_{i}\tau)$, $\B_{ni}:= \B(t_{n}+c_{i}\tau)$,
$M_{ni}:=M(t_{n}+c_{i}\tau)$ and $M_{n+1}:=M(t_{n+1})$.  \par
For the R--K method we make the following assumptions:
\begin{assumption}\label{assump_RK-method-assumptions}~
    \begin{itemize}
        \item The method has stage order $q\geq 1$ and classical order $p\geq q+1$.
        \item The coefficient matrix $(a_{ij})$ is invertible; the inverse will be
            denoted by upper indices $(a^{ij})$.
        \item The method is \emph{algebraically stable}, i.e.\ $b_{j}>0$ for $j=1,2,\dotsc,s$ and the following matrix is
            positive semi-definite:
            \begin{align}
                \label{eq_algebraic-stable}
                \big(b_{i}a_{ij} - b_{j}a_{ji} -b_{i}b_{j}\big)_{i,j=1}^{s}.
            \end{align}
        \item The method is \emph{stiffly accurate}, i.e.\ for $j=1,2,\dotsc,s$ it holds
            \begin{align}
                \label{eq_stiffly_accurate}
                b_{j}=a_{sj}, \qquad \text{and} \qquad c_{s}=1.
            \end{align}
    \end{itemize}
\end{assumption}

\bigskip
Instead of \eqref{eq_ES-ODE}, let us consider the following perturbed version of equation:
\begin{align}
  \label{eq_rk_pertubated}
  \begin{cases}
    \disp\diff \big(M\t \widetilde{\alpha} \t \big) + A\t \widetilde{\alpha}\t +
    \B\t \widetilde{\alpha} \t = M\t r\t, \\
    \disp\phantom{\diff \big(M\t \widetilde{\alpha}\t\big) + A\t
      \widetilde{\alpha}\t + \B\t }\widetilde{\alpha}(0) = \widetilde{\alpha}_0.
  \end{cases}
\end{align}

The substitution of the true solution $\tilde{\alpha}\t$ of the perturbed problem into the R--K method, yields the defects $\Delta_{ni}$ and $\delta_{ni}$, by setting $e_n = \alpha_n - \tilde{\alpha}(t_n)$, $E_{ni} = \alpha_{ni} - \tilde{\alpha}(t_n+c_i\tau)$ and $\dot{E}_{ni} = \dot{\alpha}_{ni} - \dot{\tilde{\alpha}}(t_n+c_i\tau)$, then by subtraction the following \emph{error equations} hold:

\begin{subequations}
  \begin{align}
    \label{eq_error-eq-a}
    M_{ni}E_{ni}  &= M_{n}e_{n} + \tau\sum_{j=1}^{s} a_{ij} \dot{E}_{nj}  -
    \Delta_{ni}, \qquad \text{for} \quad i=1,2,\dotsc,s, \\
    \label{eq_error-eq-b}
    M_{n+1} e_{n+1} &= M_{n} e_{n} + \tau \sum_{i=1}^{s} b_{i} \dot{E}_{ni}  -
    \delta_{n+1},
  \end{align}
\end{subequations}

\noindent where the internal stages satisfy:

\begin{align}
  \label{eq_error-eq-interStages}
  \dot{E}_{ni} + A_{ni} E_{ni} + \B_{ni} E_{ni}  = - M_{ni} r_{ni}, \qquad \text{for} \quad i=1,2,\dotsc,s,
\end{align}

\noindent with $r_{ni} := r(t_{n} + c_{i}\tau)$.

\bigskip
Now we state and prove one of the key lemmata of this paper, which provide unconditional stability for the above class of Runge--Kutta methods.
\begin{lemma}
  \label{lem_rk-stability}
  For an $s$-stage implicit Runge--Kutta method satisfying Assumption \ref{assump_RK-method-assumptions}, there exists a \linebreak $\tau_0 >0$, depending only on the constants $\mu$ and $\kappa$, \st\ for $\tau\leq\tau_0$ and $t_n=n \tau\leq T$, that the error $e_n$ is bounded by
  \begin{align*}
    \normt{e_{n}}{M_{n}}^{2} + \tau \sum_{k=1}^{n} \normt{e_{k}}{A_{k}} &\leq C
    \Bigl\{ \normt{e_{0}}{M_{0}} + \tau \sum_{k=1}^{n-1} \sum_{i=1}^{s}
    \Normstar{M_{k}r_{k}}{t_{ki}}^{2} + \tau \sum_{k=1}^{n}
    \normt{\nicefrac{\delta_k}{\tau}}{M_{k}}^{2 } \\
    & \hphantom{\leq C \Bigl\{ \lvert e_{0}\rvert_{M_{0}} .}+ \tau \sum_{k=0}^{n-1}
    \sum_{i=1}^{s} \Bigl( \normt{M_{ki}^{-1}\Delta_{ki}}{M_{ki}}^{2} +
    \normt{M_{ki}^{-1}\Delta_{ki}}{A_{ki}}^{2}\Bigr) \Bigr\},
  \end{align*}
  where $\|w\|_{\ast, j}^2=w^T(A\t+M\t)\inv w$. The constant $C$ is independent of $h, \ \tau$ and $n$, but depends on $\mu, \ \kappa, \ T$ and on the norm of the difference of the velocities.
\end{lemma}

\begin{proof}
  (a) We modify the proof of \cite{DziukLubichMansour_rksurf}, Lemma~7.1 or \cite{diss_Mansour}, Lemma~3.1, and we note that our presentation is closer to the second reference. In these works the following inequality has been shown, which also holds for the ALE setting:
  \begin{align}
    \nonumber
    \normt{e_{n+1}}{ M_{n+1}}^{2} & \leq (1+2\mu \tau) \normt{e_{n}}{M_{n}}^{2}
    + 2 \tau \sum_{i=1}^{s} b_{i}
    \braket{\dot{E}_{ni}| M_{n+1}^{-1} | M_{ni}E_{ni} + \Delta_{ni}}  \\
    \label{eq:7}
    & + \tau \normt{E_{n+1}}{M_{n+1}}^{2} + (1+3\tau) \tau
    \normbig{\nicefrac{\delta_{n+1}}{\tau}}{M_{n+1}^{-1}}^{2}.
  \end{align}
  We want to estimate the second term on the right-hand side of \eqref{eq:7}.
  Obviously the equation
  \begin{align}
    \nonumber
    \braket{\dot{E}_{ni} | M_{n+1}^{-1}| M_{ni}E_{ni} + \Delta_{ni}} & =
    \braket{\dot{E}_{ni} | M_{ni}^{-1}|M_{ni}E_{ni}+ \Delta_{ni}} \\
    \label{eq:15}
    & + \braket{\dot{E}_{ni} | M_{n+1}^{-1} - M_{ni}^{-1}|
      M_{ni}E_{ni} + \Delta_{ni}}
  \end{align}
  holds.  The second term on the right-hand side of \eqref{eq:15} can be
  estimated like (c.f.\ \cite{diss_Mansour}):
  \begin{align}
    \label{eq:16}
    \braket{\dot{E}_{ni} | M_{n+1}^{-1} - M_{ni}^{-1}|M_{ni} E_{ni} +
      \Delta_{ni}} & \leq C\Bigl\{ \normt{e_{n}}{M_{n}}^{2} + \sum_{j=1}^{s}
    \normt{E_{nj}}{M_{nj}}^{2} + \normt{\Delta_{nj}}{M_{nj}^{-1}}^{2} \Bigr\}.
  \end{align}
  (b) We have to modify the estimation of the first term on the right-hand side of
  \eqref{eq:15}.  Using the definition of internal stages~\eqref{eq_error-eq-interStages}, we have
  \begin{align}
    \nonumber \braket{\dot{E}_{ni} | M_{ni}^{-1}|M_{ni}E_{ni}+ \Delta_{ni} } =
    &- \normt{E_{ni}}{A_{ni}}^{2} - \braket{M_{ni} r_{ni}|E_{ni} +
      M_{ni}^{-1}\Delta_{ni}} \\
    \label{eq:17}
    &- \braket{E_{ni} | A_{ni}| M_{ni}^{-1} \Delta_{ni}} -
    \braket{\B_{ni}E_{ni}| E_{ni} + M^{-1}_{ni}\Delta_{ni}}.
  \end{align}
  The last term can be estimated by Lemma \ref{lemma-B_estimate} as
  \begin{align}
    \nonumber
    \lvert \braket{\B_{ni}E_{ni}|E_{ni}+M_{ni}^{-1}\Delta_{ni}}\rvert & \leq
    \lvert \braket{\B_{ni}E_{ni}|E_{ni}}\rvert + \lvert \braket{\B_{ni}E_{ni}|M_{ni}^{-1}\Delta_{ni}}\rvert\\
    \nonumber
    & \leq C\normt{E_{ni}}{M_{ni}} \normt{E_{ni}}{A_{ni}} +
    C\normt{E_{ni}}{M_{ni}} \normt{M^{-1}_{ni}\Delta_{ni}}{A_{ni}}\\
    \label{eq:18}
    & \leq C\normt{E_{ni}}{M_{ni}}^{2} +
    \frac{1}{4}\normt{E_{ni}}{A_{ni}}^{2} + C\normt{E_{ni}}{M_{ni}}^{2} +
    C \normt{M_{ni}^{-1}\Delta_{ni}}{A_{ni}}^{2}.
  \end{align}
  While the other terms can be estimated by the following inequality (shown by \cite{diss_Mansour}):
  \begin{align}
    \nonumber
    - \normt{E_{ni}}{A_{ni}}^{2} &+ \lvert \braket{M_{ni} r_{ni}|E_{ni}
      + M_{ni}^{-1}\Delta_{ni}}\rvert + \lvert \braket{E_{ni} | A_{ni}|
      M_{ni}^{-1}
      \Delta_{ni}} \rvert  \\
    \label{eq:19}
    &\leq -\frac{1}{2} \normt{E_{ni}}{A_{ni}}^{2} +
    \frac{1}{4}\normt{E_{ni}}{M_{ni}}^{2} + C \bigl(\normt{M_{ni}^{-1}
      \Delta_{ni}}{M_{ni}}^{2} + \normt{M_{ni}^{-1}
      \Delta_{ni}}{A_{ni}}^{2}\bigr).
  \end{align}
  We continue to estimate the right-hand side of \eqref{eq:17} with
  \eqref{eq:18}, \eqref{eq:19} and arrive to
  \begin{align}
    \label{eq:20}
    \braket{\dot{E}_{ni} | M_{n+1}^{-1}| M_{ni}E_{ni} + \Delta_{ni}} \leq
    -\frac{1}{4} \normt{E_{ni}}{A_{ni}}^{2} + C \bigl(\normt{E_{ni}}{M_{ni}}^{2}
    +\normt{M_{ni}^{-1} \Delta_{ni}}{M_{ni}}^{2} + \normt{M_{ni}^{-1}
      \Delta_{ni}}{A_{ni}}^{2}\bigr).
  \end{align}
  (c) Now we return to the main inequality \eqref{eq:7}, consider
  equation~\eqref{eq:17} and plug in the inequalities \eqref{eq:16} and
  \eqref{eq:20} to get
  \begin{align}
    \nonumber
    \normt{e_{n+1}}{M_{n+1}}^{2} -\normt{e_{n}}{M_{n}}^{2} + \frac{1}{4}
    \tau \sum_{i=1}^{s} b_{i} \normt{E_{ni}}{A_{ni}}^{2} & \leq C \tau\Bigl\{
    \normt{e_{n}}{M_{n}}^{2} + \sum_{j=1}^{s} \normt{E_{nj}}{M_{nj}}^{2}  +
    \Normstar{M_{nj}r_{nj}}{nj} \\
    \label{eq:1}
    & \phantom{\leq} +   \sum_{j=1}^{s}\bigl(\normt{M_{nj}^{-1}
      \Delta_{nj}}{M_{nj}}^{2} + \normt{M_{nj}^{-1}
      \Delta_{nj}}{A_{nj}}^{2}\bigr) +
    \normbig{\nicefrac{\delta_{n+1}}{\tau}}{M_{n+1}^{-1}}\Bigr\}.
  \end{align}
  (d) Next we estimate $\normt{E_{nj}}{M_{nj}}^{2}$, in \cite{diss_Mansour} one can find the estimate:
  \begin{align}
    \label{eq:6}
    \normt{E_{ni}}{M_{ni}}^{2} \leq C \Bigl( \normt{e_{n}}{M_{n}}^{2}  + \tau \sum_{j=1}^{s} a_{ij}
    \braket{\dot{E}_{nj} | E_{ni}} + \normt{M_{ni}^{-1}\Delta_{ni}}{M_{ni}}^{2} \Bigr).
  \end{align}
  We have to estimate $\braket{\dot{E}_{nj}| E_{ni}}$, with
  equation~\eqref{eq_error-eq-interStages} we get
  \begin{align}
    \label{eq:24}
    \braket{\dot{E}_{nj}|E_{ni}} = - \braket{E_{nj} | A_{nj} | E_{ni}} -
    \braket{M_{nj} r_{nj} | E_{ni}} - \braket{\B_{nj}E_{nj} | E_{ni}}.
  \end{align}
  The following inequalities can be shown easily using Young's-inequality ($\varepsilon$ will be chosen
  later) and Cauchy--Schwarz inequality:
  \begin{align*}
    -\braket{E_{nj} | A_{nj} | E_{ni}} &\leq C(\kappa) \bigl(
    \normt{E_{nj}}{A_{nj}}^{2} +
    \normt{E_{ni}}{A_{ni}}^{2}\bigr),\\
    - \braket{\B_{nj}E_{nj}| E_{ni}} &\leq \varepsilon \normt{E_{nj}}{M_{nj}}^{2}
    + \frac{1}{4\varepsilon} C(\kappa)
    \normt{E_{ni}}{A_{ni}}^{2} \\
    - \braket{M_{nj}r_{nj} | E_{ni}} & \leq C(\mu,\kappa) \Bigl(
    \frac{1}{4\varepsilon} \Normstar{M_{nj}r_{nj}}{nj}^{2} + \varepsilon \bigl(
    \normt{E_{ni}}{M_{ni}}^{2} + \normt{E_{ni}}{A_{ni}}^{2}\bigr)\Bigr).
   \end{align*}
   Using the above three inequalities to estimate \eqref{eq:24}, we get
   \begin{align}
     \label{eq:25}
     \braket{\dot{E}_{nj} | E_{ni}} \leq  C(\mu,\kappa) \Bigl( \varepsilon
     \normt{E_{ni}}{M_{ni}}^{2} + C(\varepsilon)
     \normt{E_{ni}}{A_{ni}}^{2} + \normt{E_{nj}}{A_{nj}}^{2} +
     C(\varepsilon) \Normstar{M_{nj}r_{nj}}{nj}^{2}\Bigr).
   \end{align}
   Using this for a sufficiently small $\varepsilon$ (independent of $\tau$) we can proceed by estimating \eqref{eq:6} further as
   \begin{align*}
     \normt{E_{ni}}{M_{ni}}^{2} \leq C \Bigl( \normt{e_{n}}{M_{n}}^{2} + \tau
     \sum_{j=1}^{s} a_{ij} \bigl( \normt{E_{nj}}{A_{nj}}^{2} +
     \Normstar{M_{nj}r_{nj}}{nj}^{2}\bigr) +
     \normt{M_{ni}^{-1}\Delta_{ni}}{M_{ni}}^{2} \Bigr).
   \end{align*}
   (e) Now for a sufficiently small $\tau$ we can use the above inequality to estimate
   \eqref{eq:1} to
   \begin{align*}
     \normt{e_{n+1}}{M_{n+1}}^{2} & - \normt{e_{n}}{M_{n}}^{2} + \frac{1}{8}
     \tau \sum_{i=1}^{s} b_{i} \normt{E_{ni}}{A_{ni}}^{2} \leq C\tau \Bigl\{
     \normt{e_{n}}{M_{n}}^{2} + \sum_{i=1}^{s}\Normstar{M_{ni}r_{ni}}{ni}^{2} \\
     & + \sum_{i=1}^{s} \bigl( \normt{M_{ni}^{-1} \Delta_{ni}}{M_{ni}}^{2} +
     \normt{M_{ni}^{-1}\Delta_{ni}}{A_{ni}}^{2}\bigr) +
     \normbig{\nicefrac{\delta_{n+1}}{\tau}}{M_{n+1}^{-1}}\Bigr\}.
   \end{align*}
   Summing up over $n$ and applying a discrete Gronwall inequality yields the
   desired result.
\end{proof}

\section{Error estimates for Backward Difference Formulas}
\label{section_BDF}

We apply a backward difference formula (BDF) as a temporal discretization to the ODE system \eqref{eq_ES-ODE}, coming from the ALE ESFEM space discretization of the parabolic evolving surface PDE.

In the following we extend the stability result for BDF methods of \cite{LubichMansourVenkataraman_bdsurf}, Lemma 4.1 to the case of ALE evolving surface finite element method. Apart from the properties of the ALE ESFEM the proof is based on the G--stability theory of \cite{Dahlquist} and the multiplier technique of \cite{NevanlinnaOdeh}. We will prove that the fully discrete method is unconditionally stable for the $k$-step BDF methods for $k\leq5$.

\bigskip
We recall the $k$-step BDF method for \eqref{eq_ES-ODE} with step size $\tau>0$:
\begin{equation}\label{def_BDF}
    \disp \frac{1}{\tau} \sum_{j=0}^k \delta_j M(t_{n-j})\alpha_{n-j} + A(t_n)\alpha_n + \B(t_n)\alpha_n = 0, \qquad (n \geq k),
\end{equation}
where the coefficients of the method are given by $\delta(\zeta)=\sum_{j=1}^k \delta_j \zeta^j=\sum_{\ell=1}^k \frac{1}{\ell}(1-\zeta)^\ell$, while the starting values are $\alpha_0, \alpha_1, \dotsc, \alpha_{k-1}$. The method is known to be $0$-stable for $k\leq6$ and have order $k$ (for more details, see \cite[Chapter V.]{HairerWannerII}).

\bigskip
Instead of \eqref{eq_ES-ODE} let us consider again the perturbed problem
\begin{equation}\label{eq_perturbed_ODE_BDF}
    \disp
    \begin{cases}
        \disp \diff \big(M\t \tilde{\alpha}\t\big) + A\t \tilde{\alpha}\t + \B\t \tilde{\alpha}\t = M\t r\t \\
        \disp \phantom{\diff \big(M\t \alpha\t\big) + A\t \alpha\t + \B\t }\tilde{\alpha}(0) = \tilde{\alpha}_0.
    \end{cases}
\end{equation}

By substituting the true solution $\tilde{\alpha}\t$ of the perturbed problem into the BDF method \eqref{def_BDF}, we obtain
\begin{equation*}
    \disp \frac{1}{\tau} \sum_{j=0}^k \delta_j M(t_{n-j})\tilde{\alpha}_{n-j} + A(t_n)\tilde{\alpha}_n + \B(t_n)\tilde{\alpha}_n = -d_n, \qquad (n \geq k).
\end{equation*}

By introducing the error $e_n = \alpha_n - \tilde{\alpha}(t_n)$, multiplying by $\tau$, and by subtraction we have the error equation
\begin{equation}\label{eq_BDF-error-eq}
    \disp \sum_{j=0}^k \delta_j M_{n-j} e_{n-j} + \tau A_n e_n + \tau \B_n e_n = \tau d_n, \qquad (n \geq k).
\end{equation}

\bigskip
We recall two important preliminary results.
\begin{lemma}[\cite{Dahlquist}]
    Let $\delta(\zeta)$ and $\mu(\zeta)$ be polynomials of degree at most $k$ (at least one of them of exact degree $k$) that have no common divisor. Let $\la \,.\, | \,.\, \ra$ be an inner product on $\R^N$ with associated norm $\| \,.\, \|$. If
    \begin{equation*}
        \textnormal{Re} \frac{\delta(\zeta)}{\mu(\zeta)} > 0, \qquad \textrm{for} \quad |\zeta|<1,
    \end{equation*}
    then there exists a symmetric positive definite matrix $G = (g_{ij}) \in \R^{k\times k}$ and real $\gamma_0,\dotsc,\gamma_k$ such that for all $v_0,\dotsc,v_k\in\R^N$
    \begin{equation*}
        \Big\la \sum_{i=0}^k \delta_i v_{k-i} \Big| \sum_{i=0}^k \mu_i v_{k-i}  \Big\ra = \sum_{i,j=1}^k g_{ij} \la v_i \,|\, v_j \ra - \sum_{i,j=1}^k g_{ij} \la v_{i-1} \,|\, v_{j-1} \ra + \Big\| \sum_{i=0}^k \gamma_i v_i \Big\|^2
      \end{equation*}
      holds.
\end{lemma}

Together with this result, the case $\mu(\zeta)=1-\eta\zeta$ will play an important role:
\begin{lemma}[\cite{NevanlinnaOdeh}]
    If $k\leq5$, then there exists $0\leq\eta<1$ \st\ for $\delta(\zeta)=\sum_{\ell=1}^k \frac{1}{\ell}(1-\zeta)^\ell$,
    \begin{equation*}
        \textnormal{Re} \frac{\delta(\zeta)}{1-\eta\zeta} > 0, \qquad \textrm{for} \quad |\zeta|<1.
    \end{equation*}
    The smallest possible values of $\eta$ is found to be $\eta=0, \ 0, \ 0.0836, \ 0.2878, \ 0.8160$ for $k=1,2,\dotsc,5$, \resp.
\end{lemma}

\bigskip
We now state and prove the analogous stability result for the BDF methods. Again the stability is unconditional.
\begin{lemma}\label{lemma_BDF-stability}
    For a $k$-step BDF method with $k\leq5$ there exists a $\tau_0 >0$, depending only on the constants $\mu$ and $\kappa$, \st\ for $\tau\leq\tau_0$ and $t_n=n \tau\leq T$, that the error $e_n$ is bounded by
    \begin{equation*}
        \disp |e_n|_{M_n}^2 + \tau \sum_{j=k}^n |e_j|_{A_j}^2 \leq C \tau \sum_{j=k}^n \|d_j\|_{\ast, j}^2 + C \max_{0\leq i \leq k-1} |e_i|_{M_i}^2
    \end{equation*}
    where $\|w\|_{\ast, j}^2=w^T(A\t+M\t)\inv w$. The constant $C$ is independent of $h, \ \tau$ and $n$, but depends on $\mu, \ \kappa, \ T$ and on the norm of the difference of the velocities.
\end{lemma}
\begin{proof}
Our proof follows the one of \cite{LubichMansourVenkataraman_bdsurf} Lemma 4.1.

\medskip
(a) The starting point of the proof is the following reformulation of the error equation \eqref{eq_BDF-error-eq}
\begin{equation*}
        \disp M_{n} \sum_{j=0}^k \delta_j e_{n-j} + \tau A_n e_n + \tau \B_n e_n = \tau d_n + \sum_{j=1}^k \delta_j \big(M_n-M_{n-j}\big) e_{n-j}
\end{equation*}
and using a modified energy estimate. We multiply both sides with $e_n - \eta e_{n-1}$, for $n\geq k+1$, which gives us:
\begin{equation*}
    I_n + II_n  = III_n + IV_n - V_n,
\end{equation*}
where
{\setlength\arraycolsep{.13889em}
\begin{eqnarray*}
    \disp I_n   &=& \Big\la \sum_{j=0}^k \delta_j e_{n-j} \big| M_n \big| e_n - \eta e_{n-1} \Big\ra, \\
    \disp II_n  &=& \tau \big\la e_n | A_n | e_n - \eta e_{n-1} \big\ra, \\
    \disp III_n &=& \tau \la d_n | e_n - \eta e_{n-1} \ra, \\
    \disp IV_n  &=& \sum_{j=1}^k \la e_{n-j} | M_n-M_{n-j} | e_n - \eta e_{n-1} \ra, \\
    \disp V_n   &=& \tau \la e_n | \B_n | e_n - \eta e_{n-1} \ra .
\end{eqnarray*}}

\bigskip
(b) The estimations of $I_n, \ II_n, \ III_n$ and $IV_n$ are the same as in the proof in \cite{LubichMansourVenkataraman_bdsurf}. We note that during the estimation of $III_n$ we used Young's inequality with sufficiently small ($\tau$ independent) $\eps$.

\bigskip
The new term $V_n$ is estimated using Lemma \ref{lemma-B_estimate} and Young's inequality (with sufficiently small $\eps$, independent of $\tau$):
{\setlength\arraycolsep{.13889em}
\begin{eqnarray*}
    \disp |V_n| &\leq& C \tau |e_n|_{M_n} \big( |e_n|_{A_n} +\eta|e_{n-1}|_{A_{n-1}}\big) \\
    \disp &=& C \tau |e_n|_{M_n} |e_n|_{A_n} + C\eta \tau |e_n|_{M_n} |e_{n-1}|_{A_{n-1}} \\
    \disp &\leq& \tau C \frac{1}{\eps} |e_n|_{M_n}^2 + \eps \tau |e_n|_{A_n}^2  +  \tau C \frac{1}{\eps} |e_n|_{M_n}^2 + \eps \eta^2\tau \disp |e_{n-1}|_{A_{n-1}}^2.
\end{eqnarray*}}

\bigskip
(c) Combining all estimates, choosing a sufficiently small $\eps$ (independently of $\tau$), and summing up gives, for $\tau\leq\tau_0$ and for $k\geq n+1$:
\begin{equation*}
    |E_n|_{G,n}^2 + (1-\eta)\frac{\tau}{8}\sum_{j=k+1}^n |e_j|_{A_j}^2 \leq C \tau \sum_{j=k}^{n-1} |E_j|_{G,j}^2 + C \tau \sum_{j=k+1}^n \|d_j\|_{\ast,j}^2 + C \eta^2 \tau |e_k|_{A_k}^2,
\end{equation*}
where $E_n=(e_n,\dotsc,e_{n-k+1})$, and the $|E_n|_{G,n}^2 := \sum_{i,j=1}^k g_{ij}\la e_{n-k+1} | M_n | e_{n-k+j}\ra$.

This is the same inequality as in \cite{LubichMansourVenkataraman_bdsurf}, hence we can also proceed with the discrete Gronwall inequality.

\bigskip
(d) To achieve the stated result we have to estimate the extra term $C \ \big(|e_k|_{M_k}^2 + \tau |e_k|_{A_k}^2\big)$. For that we take the inner product of the error equation for $n=k$ with $e_k$ to obtain
\begin{equation*}
    \delta_0|e_k|_{M_k}^2 + \tau |e_k|_{A_k}^2 = \tau \la d_k \,|\, e_k\ra - \sum_{j=0}^k \delta_j \la M_{k-j} e_{k-j} \,|\, e_k\ra + \tau |\la e_k \,|\, B_k \,|\, e_k\ra|.
\end{equation*}
Then the use of Lemma \ref{lemma-B_estimate} and Young's inequality (again with sufficiently small $\eps$) and \eqref{eq_mtxlemma-M}, yields
\begin{equation*}
    |e_k|_{M_k}^2 + \tau |e_k|_{A_k}^2 \leq C \tau \|d_k\|_{\ast,k}^2 + C \max_{0\leq i \leq k-1} |e_i|_{M_i}^2.
\end{equation*}
The insertion of this completes the proof.
\end{proof}

\section{Error bounds for the fully discrete solutions}\label{section_errorbounds}

We start by connecting the stability results of the previous two sections with the \co\ solution of the parabolic problem, by investigating the behaviour of the difference of the discrete numerical solution and an arbitrary projection of the true solution $u$ to the evolving surface finite element space $S_h\t$.

Then, by choosing a specific projection, we will show the optimal rate of convergence of this difference, which -- together with the stability results -- leads us to our main results. We will prove that the full discretizations, ALE evolving surface finite element method coupled with Runge--Kutta or BDF methods of the parabolic problem \eqref{eq_ES-PDE-strong-form} (and hence \eqref{eq_MD-PDE-strong} also), have an optimal order and unconditional convergence both in space and time.

\subsection{The semidiscrete residual}
We follow \cite{LubichMansourVenkataraman_bdsurf} Section 5 by setting
\begin{equation*}
    \disp P_h:H^1(\Gat) \to S_h\t\subset H^1(\Ga_h\t)
\end{equation*}
an arbitrary projection of the exact solution to the finite dimensional space $S_h\t$. Later we will choose $P_h$ to be a Ritz projection.

\bigskip
We define the finite element residual $R_h(.,t) = \sum_{j=1}^N r_j\t \chi_j(.,t)\in S_h\t$ as
\begin{equation}\label{eq_residual}
    \int_{\Ga_h} R_h \phi_h = \diff \int_{\Ga_h} P_h u \phi_h + \int_{\Ga_h} \nbgh(P_h u)\cdot \nbgh \phi_h + \int_{\Ga_h} (P_h u) (\Wh-V_h)\cdot \nbgh \phi_h - \int_{\Ga_h} (P_h u)  \amath \phi_h,
\end{equation}
where $\phi_h\in S_h\t$, and the projection of the true solution $u$ is given as
\begin{equation*}
    P_h u(.,t) = \sum_{j=1}^N \tilde{\alpha}_j\t \chi_j(.,t).
\end{equation*}

The above problem is equivalent to the ODE system with the vector $r\t=(r_j\t)\in\R^N$:
\begin{equation*}
    \diff \big(M\t \tilde{\alpha}\t\big) + A\t \tilde{\alpha}\t + \B\t \tilde{\alpha}\t = M\t r\t,
\end{equation*}
which is the perturbed ODE system \eqref{eq_rk_pertubated} and \eqref{eq_perturbed_ODE_BDF}.

\subsection{Error bounds for the time integrations}

The direct application of the stability lemmas for Runge--Kutta methods and BDF methods (Lemma \ref{lem_rk-stability} and Lemma \ref{lemma_BDF-stability}, \resp) gives optimal order error estimates between the projection $P_hu(.,t_n)$ and the fully discrete solution $U_h^n$ (ALE ESFEM combined with a temporal discretization), i.e.
\begin{equation*}
    U_h^n := \sum_{j=1}^N \alpha_j^n\chi_j(.,t),
\end{equation*}
where the vectors $\alpha^n$ are generated, either by an $s$-stage implicit Runge--Kutta method, or by a BDF method of order $k$.

\subsubsection{Implicit Runge--Kutta methods}
Now we can prove the analogous error estimation result from \cite{DziukLubichMansour_rksurf} Theorem~8.1 (\cite{diss_Mansour} Theorem~5.1).
\begin{theorem}\label{thm_R-K-errors}
    Consider the arbitrary Lagrangian Eulerian evolving surface finite element method as space discretization of the parabolic problem \eqref{eq_ES-PDE-strong-form} with time discretization by an $s$--stage implicit Runge--Kutta method satisfying Assumption \ref{assump_RK-method-assumptions}. Assume that $P_h u$ has \co\ discrete ALE material derivatives up to order $q+2$. Then there exists $\tau_0>0$, independent of $h$, \st\ for $\tau \leq \tau_0$, for the error $E_h^n=U_h^n-P_h u(.,t_n)$ the following estimate holds for $t_n=n\tau \leq T$:
    {\setlength\arraycolsep{.13889em}
    \begin{eqnarray*}
        \disp \|E_h^n\|_{L^2(\Ga_h(t_n))} &+& \Big( \tau \sum_{j=1}^n \|\nb_{\Ga_h(t_j)} E_h^j \|_{L^2(\Ga_h(t_j))}^2 \Big)^{\frac{1}{2}} \\
        \disp \leq C \tilde{\beta}_{h,q} \tau^{q+1} &+& C \Big( \tau \sum_{k=0}^{n-1} \sum_{i=1}^s \|R_h(.,t_k+c_i\tau) \|_{H^{-1}(\Ga_h(t_k+c_i\tau))}^2 \Big)^{\frac{1}{2}} + C \|E_h^0\|_{L^2(\Ga_h(t_0))},
    \end{eqnarray*}}
    where the constant $C$ is independent of $h$, but depends on $T$, and
    \begin{equation*}
        \disp \tilde{\beta}_{h,q}^2 = \int_0^T \sum_{\ell=1}^{q+2} \| (\amath)^{(\ell)} (P_h u)(.,t) \|_{L^2(\Ga_h\t)} + \sum_{\ell=1}^{q+1} \| \nb_{\Ga_h\t} (\amath)^{(\ell)} (P_h u)(.,t) \|_{L^2(\Ga_h\t)} \d t.
    \end{equation*}
    The $H^{-1}$ norm  of $R_h$ is defined as
    \begin{equation*}
        \disp \|R_h(.,t) \|_{H^{-1}(\Ga_h\t)} := \sup_{0\neq\phi_h\in S_h\t} \frac{\la R_h(.,t),\phi_h\ra_{L^2(\Ga_h\t)}}{\|\phi_h\|_{H^{1}(\Ga_h\t)}}.
    \end{equation*}
\end{theorem}

\bigskip
The version with the classical order $p$ from \cite{DziukLubichMansour_rksurf} Theorem 8.2 (or \cite[Theorem 5.2]{diss_Mansour}) also holds in the ALE case, if the stronger regularity conditions are satisfied:
{\setlength\arraycolsep{.13889em}
\begin{eqnarray*}
    \Bigg| M\t\inv \frac{\d^{k_j-1}}{\d t^{k_j-1}}\Big( A\t M\t\inv \Big) \dotsm \frac{\d^{k_1-1}}{\d t^{k_1-1}}\Big( A\t M\t\inv \Big)  \frac{\d^{\tilde{k}-1}}{\d t^{\tilde{k}-1}}\Big( M\t\tilde{\alpha}\t \Big)\Bigg|_{M\t} &\leq& \gamma, \\
    \Bigg| M\t\inv \frac{\d^{k_j-1}}{\d t^{k_j-1}}\Big( A\t M\t\inv \Big) \dotsm \frac{\d^{k_1-1}}{\d t^{k_1-1}}\Big( A\t M\t\inv \Big)  \frac{\d^{\tilde{k}-1}}{\d t^{\tilde{k}-1}}\Big( M\t\tilde{\alpha}\t \Big)\Bigg|_{A\t} &\leq& \gamma,
\end{eqnarray*}}
for all $k_j\geq1$ and $\tilde{k}\geq q+1$ with $k_1+\dotsb+k_j+\tilde{k}\leq p+1$.

\begin{theorem}\label{thm_R-K-errors_order-p}
    Consider the arbitrary Lagrangian Eulerian evolving surface finite element method as space discretization of the parabolic problem \eqref{eq_ES-PDE-strong-form}, with time discretization by an $s$-stage implicit Runge--Kutta method satisfying Assumption \ref{assump_RK-method-assumptions} with $p>q+1$. Assuming the above regularity conditions. There exists $\tau_0>0$ independent of $h$, \st\ for $\tau \leq \tau_0$, for the error $E_h^n=U_h^n-P_h u(.,t_n)$ the following estimate holds for $t_n=n\tau \leq T$:
    {\setlength\arraycolsep{.13889em}
    \begin{eqnarray*}
        \disp \|E_h^n\|_{L^2(\Ga_h(t_n))} &+& \Big( \tau \sum_{j=1}^n \|\nb_{\Ga_h(t_j)} E_h^j \|_{L^2(\Ga_h(t_j))}^2 \Big)^{\frac{1}{2}} \\
        \disp \leq C_0 \tau^{p} &+& C \Big( \tau \sum_{k=0}^{n-1} \sum_{i=1}^s \|R_h(.,t_k+c_i\tau) \|_{H^{-1}(\Ga_h(t_k+c_i\tau))}^2 \Big)^{\frac{1}{2}} + C \|E_h^0\|_{L^2(\Ga_h(t_0))},
    \end{eqnarray*}}
    where the constant $C_0$ is independent of $h$, but depends on $T$ and $\gamma$.
\end{theorem}
The proof of these theorems are using Lemma \ref{lem_rk-stability}. Otherwise they are the same as the ones in \cite{DziukLubichMansour_rksurf} (or in \cite{diss_Mansour}), but one has to work with the discrete ALE material derivatives.

\subsubsection{Backward differentiation formulae}
We prove the analogous result of \cite{LubichMansourVenkataraman_bdsurf} Theorem~5.1 (\cite{diss_Mansour} Theorem~5.3).
\begin{theorem}\label{thm_BDF-errors}
    Consider the arbitrary Lagrangian Eulerian evolving surface finite element method as space discretization of the parabolic problem \eqref{eq_ES-PDE-strong-form} with time discretization by a $k$-step backward difference formula of order $k\leq5$. Assume that $P_h u$ has \co\ discrete ALE material derivatives up to order $k+1$. Then there exists $\tau_0>0$, independent of $h$, \st\ for $\tau \leq \tau_0$, for the error $E_h^n=U_h^n-P_h u(.,t_n)$ the following estimate holds for $t_n=n\tau \leq T$:
    {\setlength\arraycolsep{.13889em}
    \begin{eqnarray*}
        \disp \|E_h^n\|_{L^2(\Ga_h(t_n))} &+& \Big( \tau \sum_{j=1}^n \|\nb_{\Ga_h(t_j)} E_h^j \|_{L^2(\Ga_h(t_j))}^2 \Big)^{\frac{1}{2}} \\
        \disp \leq C \tilde{\beta}_{h,k} \tau^{k} &+& \Big( \tau \sum_{j=1}^n \|R_h(.,t_j) \|_{H^{-1}(\Ga_h(t_j))}^2 \Big)^{\frac{1}{2}} + C \max_{0\leq i \leq k-1} \|E_h^i\|_{L^2(\Ga_h(t_i))},
    \end{eqnarray*}}
    where the constant $C$ is independent of $h$, but depends on $T$, and
    \begin{equation*}
        \disp \tilde{\beta}_{h,k}^2 = \int_0^T \sum_{\ell=1}^{k+1} \| (\amath)^{(\ell)} (P_h u)(.,t) \|_{L^2(\Ga_h(t))} \d t.
    \end{equation*}
\end{theorem}
The proof of this theorem is using Lemma \ref{lemma_BDF-stability} otherwise it is same as the one in \cite{LubichMansourVenkataraman_bdsurf} (or in \cite{diss_Mansour}), but one has to work with the discrete ALE material derivatives.

\subsection{Bound of the semidiscrete residual and the Ritz map}

We use nearly the same Ritz map introduced by \cite{LubichMansour_wave} Definition 8.1, but for the parabolic case a pointwise version suffices:
\begin{definition}
    For a given $z\in H^1(\Gat)$ there is a unique $\Pt z\in S_h\t$ \st\ for all $\phi_h\in S_h\t$, with the corresponding lift $\vphi_h=\phi_h^l$, we have
    \begin{equation}\label{def_eq_Ritz}
        \disp a_h^{\ast}(\Pt z,\phi_h) = a^\ast(z,\vphi_h) + m(z , (v_h-v)\cdot \nbg \vphi_h),
    \end{equation}
    where $a^{\ast}:=a+m$ and $a_h^{\ast}:=a_h+m_h$, to make the forms $a$ and $a_h$ positive definite. Then $\P z \in S_h^l\t$ is defined as the lift of $\Pt z$, i.e.\ $\P z = (\Pt z)^l$.
\end{definition}

Together with the definition of the Ritz map, we will also use the error estimates for the Ritz projection and for its material derivatives, see \cite[Theorem 8.2]{LubichMansour_wave} (one have to work with $z$ instead of $\mat z$) or \cite[Theorem 7.2 and 7.3]{diss_Mansour}. Basically the original proof suffices for the error estimates for the ALE case as well. Except, one has to revise the following estimate.

\begin{lemma}
    The error between the material velocity $v$ and the discrete lifted material velocity $v_h$ on the smooth surface can be estimated as
    \begin{equation*}
        \| (\amath)^{(\ell)} (v-v_h) \|_{L^\infty(\Ga)} + h \| \nbg (\amath)^{(\ell)} (v-v_h) \|_{L^\infty(\Ga)} \leq c h^2,
    \end{equation*}
    for $\ell\geq0$, where $(\amath)^{(\ell)}$ denotes the $\ell$-th discrete ALE material derivative
\end{lemma}

\begin{proof}
    The key trick of the proof is expressing $w_h$ by following a material point, see
    \cite{diss_Mansour} equation (6.6), and that the normal component of $v$ and $w$ is equal. We also use the fact that (using the geometric estimates, Lemma \ref{lemma_geometric-est}) it is easy to prove the same estimate for the ALE velocity $w_h$.

    \bigskip
    (a) For $\ell=0$: the velocity $v_h$ can be expressed as

    \begin{equation*}
        v_h + w_h^{\ale} = w_h =(\pr-d\wein)\Wh - (\pa_t d) \nu - d \pa_t\nu,
    \end{equation*}
    where $-(\pa_td)\nu$ is just the normal component of $v$, denoted by $v\normal$. The superscript $\ale$ denotes the purely tangential ALE component, i.e.\ $w^{\ale}=w-v$. Further by $I_h$ we denote the finite element interpolation operator (which has its usual estimations). Then by expressing $v_h$ from above and using \eqref{eq_v-w_tangent} (i.e.\ $v\normal=w\normal$), we have
    {\setlength\arraycolsep{.13889em}
    \begin{eqnarray}
        v-v_h &=& v - v\normal - \pr W_h + w_h^{\ale} + d(\wein W_n+\pa_t\nu) \nonumber \\
              &=& (w-w\normal - \pr \Wh) + (w_h^{\ale}-w^{\ale}) + d(\wein\Wh+\pa_t\nu) \nonumber \\
              &=& \pr (w - I_h w) + (I_h w^{\ale}-w^{\ale}) + d(\wein\Wh+\pa_t\nu) \label{eq_v_diff_express}.
    \end{eqnarray}}
    Then we can estimate as
    {\setlength\arraycolsep{.13889em}
    \begin{eqnarray*}
        |v-v_h| &\leq& |\pr (w - I_h w)| + |I_h w^{\ale}-w^{\ale}| + |d(\wein\Wh+\pa_t\nu)| \leq ch^2.
    \end{eqnarray*}}
    Here the first two parts were estimated by interpolation estimates (for piecewise linear interpolants), while the last part was estimated using the geometric estimates of Lemma \ref{lemma_geometric-est}.

    We use the fact that $\nbg d=0$ and \eqref{eq_v_diff_express}, then estimate as
    {\setlength\arraycolsep{.13889em}
    \begin{eqnarray*}
        |\nbg(v-v_h)| \leq c|w - I_h w| + c|\nbg(w - I_h w)| + |\nbg(I_h w^{\ale}-w^{\ale})| + ch^2 \leq c h.
    \end{eqnarray*}}

    \bigskip
    (b) For $\ell=1$, we have $\amath \chi_j^l=0$ (transport property). Again Lemma \ref{lemma_geometric-est} implies

    {\setlength\arraycolsep{.13889em}
    \begin{eqnarray*}
        |\amath(v-v_h)| &\leq& |(\amath \pr) (w - I_h w)| + |\pr (\amath w - I_h \amath w)| + |I_h \amath w^{\ale} - \amath w^{\ale}| \\
        &+& |(\amath d)(\wein\Wh+\pa_t\nu)| + |d \amath (\wein\Wh+\pa_t\nu)| \leq c h^2.
    \end{eqnarray*}}

    For the gradient part we have $\nbg \amath d=0$, and we obtain
    {\setlength\arraycolsep{.13889em}
    \begin{eqnarray*}
        |\nbg\amath(v-v_h)| &\leq& c|w - I_h w| + c|\nbg(w - I_h w)| + c|\amath(w - I_h w)| \\
                      &+&    c|\nbg(\amath w - I_h \amath w)| + |\nbg(I_h \amath w^{\ale}-\amath w^{\ale})|+ ch^2 \leq c h.
    \end{eqnarray*}}

    (c) For $\ell>1$ the proof is analogous.
\end{proof}

\bigskip
We now replace the projection $P_h$ in the definition of $R_h$ \eqref{eq_residual}, with the Ritz map $\Pt$, and show its optimal, second order convergence.

\begin{theorem}(Bound of the semidiscrete residual)\label{thm_res-bound}
    Let $u$, the solution of the parabolic problem, be sufficiently smooth. Then there exists a constant $C>0$ and $h_0>0$, \st\ for all $h\leq h_0$ and $t\in[0,T]$, the finite element residual $R_h$ of the Ritz map is bounded by
    \begin{equation*}
        \|R_h\|_{H\inv(\Ga_h\t)} \leq C h^2.
    \end{equation*}
\end{theorem}

\begin{proof}
    (a) We start by applying the discrete ALE transport property to the residual equation \eqref{eq_residual} for $P_h=\Pt$:
    {\setlength\arraycolsep{.13889em}
    \begin{eqnarray*}
         m_h(R_h,\phi_h) &=& \diff m_h(\Pt u,\phi_h) + a_h(\Pt u,\phi_h) - m_h(\Pt u,\amath\phi_h) + m_h(\Pt u,(\Wh-V_h)\cdot\nbgh\phi_h) \\
         &=& m_h(\amath\Pt u,\phi_h) + a_h(\Pt u,\phi_h) + g_h(\Wh;\Pt u,\phi_h) + m_h(\Pt u,(\Wh-V_h)\cdot\nbgh\phi_h).
    \end{eqnarray*}}

    \medskip
    (b) We continue by the transport property with discrete ALE material derivatives from Lemma \ref{lemma_transport-prop}, but for the ALE weak form (from Lemma \ref{lemma_ALE-ES-PDE-weak}), with $\vphi:=\vphi_h=(\phi_h)^l$:
    {\setlength\arraycolsep{.13889em}
    \begin{eqnarray*}
         0 &=& \diff m(u,\vphi_h) + a(u,\vphi_h) - m(u,\amat\vphi_h) + m(u,(w-v)\cdot\nbg\vphi_h) \\
         &=& m(\amath u,\vphi_h) + a(u,\vphi_h) + g(w_h;u,\vphi_h) + m(u,(w-v)\cdot\nbg\vphi_h) -
         m(u,\amat\vphi_h-\amath\vphi_h).
    \end{eqnarray*}}

    For the last term we have
    \begin{equation*}
        \amat\vphi_h-\amath\vphi_h = (w-w_h) \cdot \nbg \vphi_h,
    \end{equation*}
    hence the last two terms can be collected as $\disp m(u,(w_h-v)\cdot\nbgh\phi_h)$.

    \medskip
    (c) Subtraction of the two equations yields
    {\setlength\arraycolsep{.13889em}
    \begin{eqnarray*}
        m_h(R_h,\phi_h) &=& m_h(\amath\Pt u,\phi_h) - m(\amath u,\vphi_h) \\
                        &+& g_h(\Wh;\Pt u,\phi_h) - g(w_h;u,\vphi_h) \\
                        &+& a^\ast_h(\Pt u,\phi_h) - a^\ast(u,\vphi_h) \\
                        &-& \big(m_h(\Pt u,\phi_h) - m(u,\vphi_h)\big) \\
                        &+& m_h(\Pt u,(\Wh-V_h)\cdot\nbgh\phi_h) - m(u,(w_h-v)\cdot\nbg\vphi_h).
    \end{eqnarray*}}
    By using the definition of the Ritz map, and then collecting the terms as
    {\setlength\arraycolsep{.13889em}
    \begin{eqnarray*}
                        & & m(u,(v_h-v)\cdot\nbg\vphi_h) + \\
                        &+& m_h(\Pt u,(\Wh-V_h)\cdot\nbgh\phi_h) - m(u,(w_h-v)\cdot\nbg\vphi_h) = \\
                        &=& m_h(\Pt u,(\Wh-V_h)\cdot\nbgh\phi_h) - m(u,(w_h-v_h)\cdot\nbg\vphi_h),
    \end{eqnarray*}}
    we finally obtain the following expression for the residual:
    {\setlength\arraycolsep{.13889em}
    \begin{eqnarray*}
        m_h(R_h,\phi_h) &=& m_h(\amath\Pt u,\phi_h) - m(\amath u,\vphi_h) \\
                        &+& g_h(\Wh;\Pt u,\phi_h) - g(w_h;u,\vphi_h) \\
                        &-& \big(m_h(\Pt u,\phi_h) - m(u,\vphi_h)\big) \\
                        &+& m_h(\Pt u,(\Wh-V_h)\cdot\nbgh\phi_h) - m(u,(w_h-v_h)\cdot\nbg\vphi_h).
    \end{eqnarray*}}

    \medskip
    (d) We estimate these pairs separately. By applying Lemma \ref{lemma_estiamtion-of-new-form} and the error estimate for the Ritz map  c.f.\ \cite{diss_Mansour} Theorem 7.2 and 7.3, there follows
    {\setlength\arraycolsep{.13889em}
    \begin{eqnarray*}
        m_h(\Pt u,(\Wh-V_h)\cdot\nbgh\phi_h) &-& m(\P u,(w_h-v_h)\cdot\nbg\vphi_h) \\
        &+& m(\P u-u,(w_h-v_h)\cdot\nbg\vphi_h) \leq C h^2\|\vphi_h\|_{H^1(\Ga)}.
    \end{eqnarray*}}
    Finally, the other pairs can be estimated by the same arguments (in fact they can be bounded by $C h^2\|\vphi_h\|_{L^2(\Ga\t)}$).
\end{proof}

\subsection{Error of the full ALE discretizations}

We compare the lifted fully discrete numerical solution $u_h^n:=(U_h^n)^l$ with the exact solution $u(.,t_n)$ of the evolving surface PDE \eqref{eq_ES-PDE-strong-form} (or the moving domain PDE \eqref{eq_MD-PDE-strong}), where $U_h^n = \sum_{j=1}^N \alpha_j^n\chi_j(.,t)$, where the vectors $\alpha^n$ are generated by the Runge--Kutta or BDF method.

\bigskip
Now we state and prove the main results of this paper.
\begin{theorem}[ALE ESFEM and R--K]
    Consider the arbitrary Lagrangian Eulerian evolving surface finite element method as space discretization of the parabolic problem \eqref{eq_ES-PDE-strong-form} with time discretization by an $s$--stage implicit Runge--Kutta method satisfying Assumption \ref{assump_RK-method-assumptions}. Let $u$ be a sufficiently smooth solution of the problem and assume that the initial value is approximated as
    \begin{equation*}
        \disp \|u_h^0 - (\P u)(.,0)\|_{L^2(\Ga(0))} \leq C_0 h^2.
    \end{equation*}
    Then there exists $h_0>0$ and $\tau_0>0$, \st\ for $h\leq h_0$ and $\tau \leq \tau_0$, the following error estimate holds for $t_n=n\tau \leq T$:
    \begin{equation*}
        \disp \|u_h^n - u(.,t_n)\|_{L^2(\Ga(t_n))} + h\Big( \tau \sum_{j=1}^n \|\nb_{\Ga(t_j)} u_h^j - \nb_{\Ga(t_j)} u(.,t_j)\|_{L^2(\Ga(t_j))}^2 \Big)^{\frac{1}{2}} \leq C \big( \tau^{q+1} + h^2 \big).
    \end{equation*}
    The constant $C$ is independent of $h, \ \tau$ and $n$.

    \medskip
    Assuming that we have more regularity: conditions of Theorem \ref{thm_R-K-errors_order-p} are additionally satisfied, then we have $p$ instead of $q+1$.
\end{theorem}

\begin{theorem}[ALE ESFEM and BDF]
    Consider the arbitrary Lagrangian Eulerian evolving surface finite element method as space discretization of the parabolic problem \eqref{eq_ES-PDE-strong-form} with time discretization by a $k$-step backward difference formula of order $k\leq5$. Let $u$ be a sufficiently smooth solution of the problem and assume that the starting values are satisfying
    \begin{equation*}
        \disp \max_{0 \leq i \leq k-1} \| u_h^i - (\P u)(.,t_i) \|_{L^2(\Ga(0))} \leq C_0 h^2.
    \end{equation*}
    Then there exists $h_0>0$ and $\tau_0>0$, \st\ for $h\leq h_0$ and $\tau \leq \tau_0$, the following error estimate holds for $t_n=n\tau \leq T$:
    \begin{equation*}
        \disp \|u_h^n - u(.,t_n)\|_{L^2(\Ga(t_n))} + h\Big( \tau \sum_{j=1}^n \|\nb_{\Ga(t_j)} u_h^j - \nb_{\Ga(t_j)} u(.,t_j)\|_{L^2(\Ga(t_j))}^2 \Big)^{\frac{1}{2}} \leq C \big( \tau^{k} + h^2 \big).
    \end{equation*}
    The constant $C$ is independent of $h, \ \tau$ and $n$.
\end{theorem}

\begin{proof}
    The global error is decomposed into two parts
    \begin{equation*}
        \disp u_h^n - u(.,t_n) = \Big(u_h^n - (\P u)(.,t_n)\Big) + \Big((\P u)(.,t_n) - u(.,t_n)\Big),
    \end{equation*}
    and the terms are estimated by previous results.

    \bigskip
    The first term is estimated by our results for Runge--Kutta or BDF methods: Theorem \ref{thm_R-K-errors} or \ref{thm_BDF-errors}, \resp, together with the residual bound Theorem \ref{thm_res-bound}, and by Theorem 7.2 and 7.3 from \cite{diss_Mansour} (or Theorem 8.2 of \cite{LubichMansour_wave}).

    \bigskip
    The second part is estimated again by the error estimates for the Ritz projection \cite{diss_Mansour} (or \cite{LubichMansour_wave} Theorem 8.2).
\end{proof}

\section{Numerical experiments}
\label{section_numerics}

We present numerical experiments for an evolving surface parabolic problem discretized by the original and the ALE evolving surface finite elements coupled with various time discretizations. The fully discrete methods were implemented in Matlab, while the initial triangulations were generated using DistMesh (\cite{distmesh}).

\medskip
The ESFEM and the ALE ESFEM case were integrated by identical codes, except the involvement of the nonsymmetric $B$ matrix and the evolution of the surface. The ODE system giving the normal movement (see \eqref{eq_normal-movement-ODE} below) was solved by the exact same time discretization method as the PDE problem (with the same step size), while the ALE map is given in \eqref{eq_ALE-map}.

\bigskip
To illustrate our theoretical results we choose a problem which was intensively investigated in the literature before, see \cite{BarreiraElliottMadzvamuse_patternformation}. Specially for ALE approach see \cite{ElliottStyles_ALEnumerics}, \cite{ElliottVenkataraman_ALEdiscrete}. We consider the evolving surface parabolic PDE \eqref{eq_ES-PDE-strong-form} over the closed surface $\Ga\t$ given by the zero level set of the distance function
\begin{equation*}
    d(x,t) := x_1^2 + x_2^2 + A\t^2 G\Big( \frac{x_3^2}{L\t^2} \Big)-A\t^2, \quad \textrm{i.e.,} \quad \Ga\t:=\{x\in \R^3 \ \big| \ d(x,t)=0\}.
\end{equation*}
Here the functions $G$, $L$ and $A$ are given as
{\setlength\arraycolsep{.13889em}
\begin{eqnarray*}
    G(s)&=&200s\big( s - \frac{199}{200}\big),\\
    L(t)&=&1 + 0.2\sin(4\pi \ t),\\
    A(t)&=&0.1 + 0.05\sin(2\pi \ t).
\end{eqnarray*}}
The velocity $v$ is the normal velocity of the surface defined by the differential equation (formulated for the nodes):
\begin{equation}\label{eq_normal-movement-ODE}
    \diff a_j = V_j \nu_j, \qquad V_j=\frac{-\pa_t d(a_j,t)}{|\nb d(a_j,t)|}, \quad \nu_j=\frac{\nb d(a_j,t)}{|\nb d(a_j,t)|}.
\end{equation}

The righthand-side $f$ is chosen as to have the function $u(x,t)=e^{-6t}x_1x_2$ to be the true solution.

Finally we give the applied ALE movement (from \cite{ElliottStyles_ALEnumerics} and \cite{ElliottVenkataraman_ALEdiscrete}):
\begin{equation}\label{eq_ALE-map}
    (a_i\t)_1= (a_0\t)_1 \frac{A\t}{A(0)}, \quad (a_i\t)_2= (a_0\t)_2 \frac{A\t}{A(0)}, \quad (a_i\t)_3= (a_0\t)_3 \frac{L\t}{L(0)},
\end{equation}
hence $d(a_i \t ,t)=0$ for every $t\in[0,T]$, for $i=1,2,\dotsc,N$.


\bigskip
In the following we compare the ALE and non-ALE methods with three spatial refinements, and integrate the evolving surface PDE with various time discretizations, with a fixed time step $\tau$, until $T=0.6$. There we compute the error vector $e\in\R^N$, representing $e_h(x,t) := u_h(x,T)-u(x,T)$ ($T = n \tau$). We also compute the following norm and seminorm of it
\begin{equation*}
   |e_h|_M = \big(e^T M(T) e\big)^{\frac{1}{2}}, \qquad  |e_h|_A = \big(e^T A(T) e\big)^{\frac{1}{2}}
\end{equation*}
which by \eqref{eq_normdefs} correspond to the $L^2$ norms of $e_h$ and $\nbgh e_h$, respectively.

The following plots show the above error norms (left $M$-norm, right $A$-norm) plotted against the time step size $\tau$ (on logarithmic scale), different error curves are representing different spatial discretizations.
l
\bigskip
In the first experiment we used the implicit Euler method as a time discretization. Figure \ref{fig: ESFEM_IE} and \ref{fig: ALE_ESFEM_IE} show the errors obtained by the backward Euler method. The convergence in time can be seen (note the reference line), while for sufficiently small $\tau$ the spatial error is dominating, in agreement with the theoretical results.
\begin{figure}[h!]
  \centering
  \includegraphics[scale=0.75]{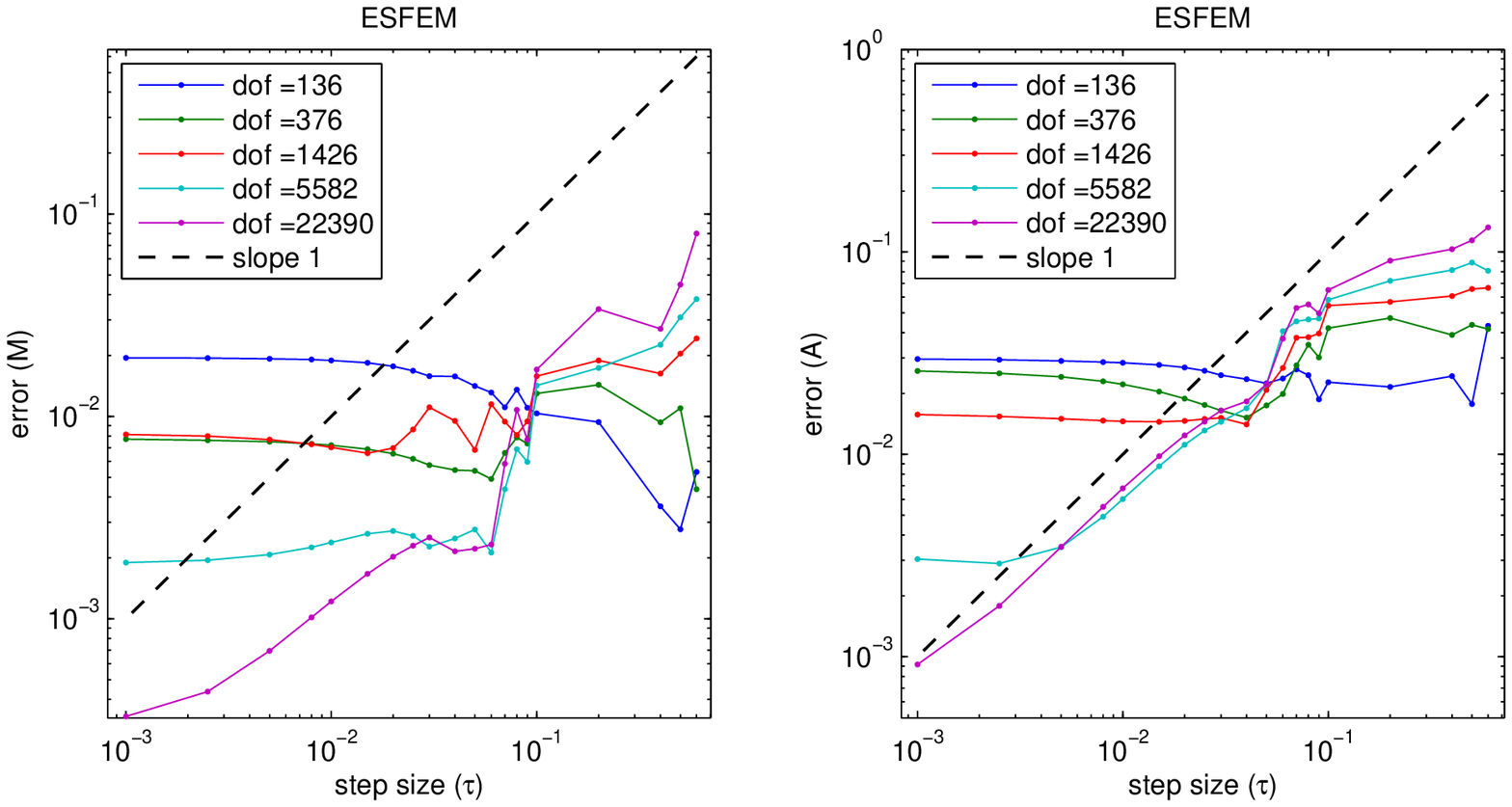}\\
  \caption{Errors of the ESFEM and the implicit Euler method}\label{fig: ESFEM_IE}

  \centering
  \includegraphics[scale=0.75]{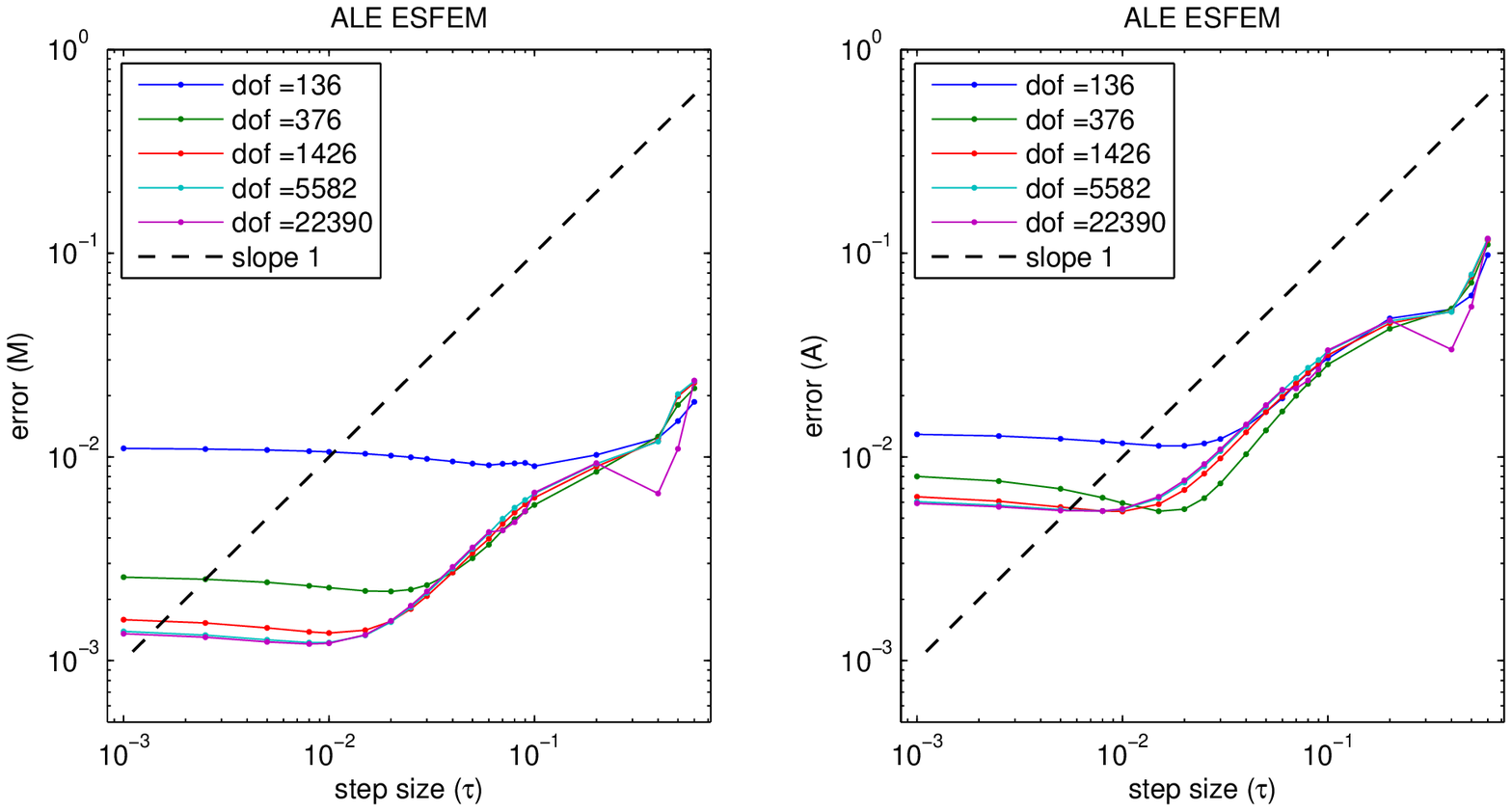}\\
  \caption{Errors of the ALE ESFEM and the implicit Euler method}\label{fig: ALE_ESFEM_IE}
\end{figure}

The figures show that the erros in the ALE ESFEM are significantly smaller than for the non-ALE.

%
%
%
%
%
%

\section*{Acknowledgement}

The authors would like to thank Prof. Christian Lubich for the invaluable discussions on the topic, and for his encouragement and help during the preparation of this paper.

\pagebreak
\bibliographystyle{alpha}
\bibliography{literature}

\end{document}